\documentclass[11pt]{amsart}   	
\usepackage{geometry}             
\usepackage[dvipsnames]{xcolor}   		             		
\usepackage{graphicx}				
\usepackage{amssymb,mathrsfs}
\usepackage{amsthm,amsmath,stmaryrd}
\usepackage{tikz}
\usepackage{tikz-cd}
\usepackage{enumerate}
\usepackage[headings]{fullpage}
\usepackage{bm}
\usepackage{mathtools}
\usepackage[all]{xy}
\usepackage{caption}
\usepackage[euler-digits]{eulervm}
\usepackage[normalem]{ulem}
\usepackage{ marvosym }

\tikzset{
  commutative diagrams/.cd, 
  arrow style=tikz, 
  diagrams={>=stealth}
}

\newtheorem{innercustomthm}{Theorem}
\newenvironment{customthm}[1]
  {\renewcommand\theinnercustomthm{#1}\innercustomthm}
  {\endinnercustomthm}

  \makeatletter
\def\@tocline#1#2#3#4#5#6#7{\relax
  \ifnum #1>\c@tocdepth 
  \else
    \par \addpenalty\@secpenalty\addvspace{#2}%
    \begingroup \hyphenpenalty\@M
    \@ifempty{#4}{%
      \@tempdima\csname r@tocindent\number#1\endcsname\relax
    }{%
      \@tempdima#4\relax
    }%
    \parindent\z@ \leftskip#3\relax \advance\leftskip\@tempdima\relax
    \rightskip\@pnumwidth plus4em \parfillskip-\@pnumwidth
    #5\leavevmode\hskip-\@tempdima
      \ifcase #1
       \or\or \hskip 1em \or \hskip 2em \else \hskip 3em \fi%
      #6\nobreak\relax
    \dotfill\hbox to\@pnumwidth{\@tocpagenum{#7}}\par
    \nobreak
    \endgroup
  \fi}
\makeatother

\usetikzlibrary{calc}
\usetikzlibrary{fadings}
\usetikzlibrary{decorations.pathmorphing}
\usetikzlibrary{decorations.pathreplacing}

\newcounter{marginnote}
\DeclareMathAlphabet{\mathpzc}{OT1}{pzc}{m}{it}

\usepackage[backref=page]{hyperref}
\hypersetup{
  colorlinks   = true,          
  urlcolor     = violet,          
  linkcolor    = blue,          
  citecolor   = violet             
}

\newtheorem{theorem}{Theorem}[subsection]
\newtheorem{corollary}[theorem]{Corollary}
\newtheorem{lemma}[theorem]{Lemma}
\newtheorem{proposition}[theorem]{Proposition}

\newtheorem{quasi-theorem}[theorem]{Quasi-Theorem}

\theoremstyle{definition}
\newtheorem{definition}[theorem]{Definition}

\newtheorem{remark}[theorem]{Remark}

\newtheorem{problem}[theorem]{Problem}

\newtheorem{construction}[theorem]{Construction}
\newtheorem{terminology}[theorem]{Terminology}
\newtheorem{example}[theorem]{Example}
\newtheorem{blank remark}[theorem]{}

\newtheorem{not1}[theorem]{Notation}

\newcommand{\A}{{\mathbb{A}}}

\newcommand{\PP}{\mathbb{P}}         
\newcommand{\QQ} {{\mathbb Q}}		
\newcommand{\RR} {{\mathbb R}}		
\newcommand{\ZZ} {{\mathbb Z}}

\def\DR{\mathsf{DR}_g}
\def\sV{\mathsf{V}}
\def\TC{\mathsf{TC}_g}

\newcommand{\cal}{\mathcal}

\def\cM{{\cal M}}

\newcommand{\plC}{\scalebox{0.8}[1.3]{$\sqsubset$}}

\newcommand{\Mbar}{\overline{\cM}\vphantom{\cM}}

\def\trop{\mathrm{trop}}

\makeatletter
\def\blfootnote{\xdef\@thefnmark{}\@footnotetext}
\makeatother

\title[Intersections on blowups of $\Mbar_{g,n}$]{A case study of intersections on blowups of the moduli of curves}

\date{}

\author{Sam Molcho {\it \&} Dhruv Ranganathan}

\address{Sam Molcho \\ Department of Mathematics\\
ETH Z\"urich, Switzerland}
\email{\href{mailto:samouil.molcho@math.ethz.ch}{samouil.molcho@math.ethz.ch}}

\address{Dhruv Ranganathan \\ Department of Pure Mathematics {\it \&} Mathematical Statistics\\
University of Cambridge, UK}
\email{\href{mailto:dr508@cam.ac.uk}{dr508@cam.ac.uk}}

\begin{document}

\maketitle

\begin{abstract}
We explain how logarithmic structures select principal components in an intersection of schemes. These manifest in Chow homology and can be understood using strict transforms under logarithmic blowups. Our motivation comes from Gromov--Witten theory. The \textit{toric contact cycles} in the moduli space of curves parameterize curves that admit a map to a fixed toric variety with prescribed contact orders. We show that they are intersections of virtual strict transforms of double ramification cycles in blowups of the moduli space of curves. We supply a calculation scheme for the virtual strict transforms, and deduce that toric contact cycles lie in the tautological ring of the moduli space of curves. This is a higher dimensional analogue of a result of Faber--Pandharipande. The operational Chow rings of Artin fans play a basic role, and are shown to be isomorphic to rings of piecewise polynomials on associated cone complexes. The ingredients in our analysis are Fulton's blowup formula, Aluffi's formulas for Segre classes of monomial schemes, piecewise polynomials, and degeneration methods. A model calculation in toric intersection theory is treated without logarithmic methods and may be read independently. 
\end{abstract}

\vspace{0.3in}
\setcounter{tocdepth}{1}
\tableofcontents

%

\section{Introduction}

Logarithmic geometry plays a central role in the construction of compact moduli spaces in algebraic geometry. In recent years, researchers have applied these techniques to study a range of questions in enumerative geometry and the moduli space of curves, and our work is especially motivated by~\cite{Herr,Hol17,HPS19,MW17,MR20,MW18,NR19,R19}. These results bring into sharp focus a basic phenomenon: logarithmic moduli spaces arise naturally as an infinite collection of schemes assembled in an inverse system. The system is analogous to the system of equivariant compactifications of a fixed torus. It occurs frequently that the inverse system does not contain a minimal model, and even when it does, expected intersection theory statements may fail on the minimal model but hold on finer members of the system. {Working with the whole inverse system reveals additional structure. \footnote{We are by no means the first to notice this, as Mumford puts it: ``the non-uniqueness (of compactification) gives one freedom to seek for the most elegant solutions in any particular case.", see~\cite{Mum72}}} Instances include the following:

\begin{enumerate}[(i)]
\item The \textbf{logarithmic Picard group} is not {representable} by a scheme with a logarithmic structure; its geometry is fully captured by an inverse system of toroidal compactifications of semiabelian schemes. In what appears to be an analogous situation, no minimal model of the \textbf{logarithmic Hilbert scheme of curves} is known~\cite{MR20,MW18}.
\item {Toroidal compactifications of the \textbf{moduli space of Abelian varieties} appear in an inverse system without a minimal representative; the situation is analogous to that of the logarithmic Picard group~\cite{KKN20}.}
\item The naive \textbf{degeneration formula} in logarithmic Gromov--Witten theory fails to hold for the minimal model of the space of stable maps, but holds on sufficiently fine models~\cite{R19}. The situation in logarithmic Donaldson--Thomas theory is expected to be similar~\cite{MR20}. 
\item The naive \textbf{product formulas} in logarithmic Abel--Jacobi theory and Gromov--Witten theory for logarithmic and rubber geometries fail to hold on the space of stable maps, but hold on sufficiently fine models~\cite{Herr,HPS19,NR19}. 
\end{enumerate}

This text provides a framework to understand these phenomena. We illustrate, by means of a detailed case study on the moduli space of curves and the problem (iv) above, the intersection theory of cycles on such inverse systems. We deduce that certain natural cohomology classes, the \textit{toric contact cycles}, lie in the tautological ring of the moduli space of curves, see Theorem~\ref{thm: toric-contact-cycles}. 

We first outline the basic geometry and then state the main results.
 

\subsection{The idealized problem} We contemplate a smooth, non-compact moduli space $U$ of non-degenerate objects, and a system of logarithmic compactifications $U \subset Y_\alpha$, indexed by a combinatorial datum $\alpha$, usually a choice of polyhedral structure on a topological space. In the ideal scenario, each compactification contains $U$ as the complement of a normal crossings divisor. Any two compactifications $Y_\alpha,Y_\beta$ in the system are dominated by a third $Y_\gamma$, via a birational morphism: an iteration of blowups and root stack constructions on strata in the complement of $U$. 

We examine intersections of cycles that occur on spaces in the inverse system. Intersection theory requires us to either pick a compactification arbitrarily or work with the system of all compactifications simultaneously; our paper concerns the comparison between the two. In the latter case, the Chow homology and cohomology groups may be defined formally as limits and colimits of the corresponding groups of the system. A pair $(Y_\alpha,V_\alpha)$ of a compactification of $U$ and a cycle $V_\alpha$ on it determines a compatible system of classes in the limit Chow groups, by pulling back along blowups. By cardinality considerations, comparatively few classes in limit Chow homology arise in this way. 

Nevertheless, in logarithmic moduli problems, birational invariance statements imply that many classes of interest do arise by pullback of a cycle along \textit{some} element~\cite{AW,MR20,R19}. If this holds, we are left to understand the intersection product of cycles $V_\alpha$ and $V_\beta$ on different models. There are two possibilities for this intersection product, and we illustrate the geometry in the simplest situation. 

\subsection{The toric specialization} Let $U$ be a torus, and $V_1,V_2$ two subvarieties intersecting transversely. On any smooth toric compactification $U\subset Y$, we may form the closures  $$\overline V^Y_1, \ \  \overline V^Y_2 \subset Y$$ to obtain
\[
V_1\cdot_Y V_2:=\overline V^Y_1\cdot \overline V^Y_2 \ \ \textrm{in } \mathsf{CH}_{\star}(Y). 
\]
Compact models are related by blowups and blow-downs along smooth centers. Given a blowup $\pi:Y'\to Y$, the classes $V_1\cdot_Y V_2$ and $V_1\cdot_{Y'} V_2$ are typically \textit{not} related either by pushforward or pullback. The closures in different compactifications are instead related by \textit{strict transforms}.

By an elementary but crucial argument, if the pair of cycles is fixed, there is a \textit{stable} answer: if $Y$ is any \textit{sufficiently fine} toric compactification of $U$, the intersection product is a well-defined cycle class, independent of the choice of compactification. In other words, if $Y$ is replaced by a blowup along a smooth center, the class is pulled back. The question of whether a given compactification is fine enough is a tropical question, answered by a theorem of Tevelev~\cite{Tev07}: the tropicalizations of $V_1$ and $V_2$ must each be unions of cones in the fan of $Y$. 

We view the stable class $[V_1\cdot_Y V_2]$ as a prototype for the \textit{logarithmic intersection product} and view it as an element in the direct limit of Chow cohomology groups of all toric compactifications of $U$. 

\noindent
\textbf{A Question.} {\it How does the class of the logarithmic intersection product differ from the ordinary intersection product on an arbitrarily chosen model $Y$ that is not sufficiently fine?}

We study this problem in a more general context, replacing toric varieties and cycles with logarithmically flat schemes, and examine intersections with logarithmic local complete intersection morphisms. Our analysis involves four tools: (i) the resolution package for toroidal schemes and morphisms~\cite{AK00,AKMW,ALT18,Mol16}, (ii) Fulton's blowup formula~\cite{Ful84}, (iii) Aluffi's formula for Segre classes of monomial subschemes~\cite{Alu13}, and (iv) the ring of piecewise polynomial functions on a logarithmic scheme~\cite{Br96,Pay06}. The ingredients are combined by using Artin fans~\cite{ACMUW, ACMW, AW}.

\subsection{The moduli space of curves} After the formalism has been set up, we illustrate the geometry by a detailed case-study on the moduli space of curves. Fix a toric variety $X$ of dimension $r$ and consider smooth $n$-pointed curves of genus $g$ with a map to $X$, where each marked point has a fixed contact order with each toric divisor. All contact points with the boundary will be marked and two maps are considered equivalent if they are translates under the torus action. Logarithmic geometry gives rise to a compactification of this cycle, known as the \textit{toric contact cycle}. 

\begin{customthm}{A}\label{thm: toric-contact-cycles}
The toric contact cycles are contained in the tautological ring of the moduli space of curves. 
\end{customthm}

The class arises naturally on the moduli space of logarithmic stable maps to $X$, see Remark~\ref{rem: tc-gwt}. When $X$ is $\mathbb P^1$ the statement was proved by Faber--Pandharipande, and their result is an ingredient in ours~\cite{FP}. The result might be viewed as saying that there is hope in finding a \textit{formula} for the class. Elementary examples in genus $1$ and $2$, can be computed using our methods, however a complete formula demands further study.

Consider the inverse system of blowups of $\Mbar_{g,n}$ along boundary strata. The space of stable maps to an unparameterized $\mathbb P^1$ -- relative to $0$ and $\infty$ -- lifts to a system of classes on the blowups, together called the logarithmic double ramification cycle, compatible under pushforward. The toric contact cycles are not products of double ramification cycles on $\Mbar_{g,n}$, but are products on sufficiently refined blowups. On such a blowup 
\[
\Mbar_{g,n}'\to \Mbar_{g,n}
\]
the lifted cycles can be understood as \textit{virtual strict transforms} of the double ramification cycles. We control virtual strict transforms and calculate them algorithmically in terms of the standard $\Mbar_{g,n}$ package -- tautological classes, strata of the double ramification cycle, and normal bundles to the exceptional strata. Fulton computes ordinary strict transforms via Segre classes, and our virtual strict transforms are calculated via virtual Segre classes. The latter are virtual pullbacks of Segre classes defined in a universal situation. This type of analysis is expected to occur frequently in nature, see Section~\ref{sec: related-problems}.

\subsection{Piecewise polynomials} When working with cycles on the inverse system of toric compactifications of a fixed torus, the intersection theory can be understood in combinatorial terms, and this is presented in Section~\ref{sec: toric-int-thy}. The generalization to logarithmic schemes follows this path. 

The first observation is that logarithmic schemes come equipped with a natural collection of cohomology classes. An \textit{Artin cone} is the stack quotient of an affine toric variety by its dense torus and an \textit{Artin fan} is a logarithmic algebraic stack that is logarithmically \'etale over a point and admits a strict cover by Artin cones. Under mild assumptions, every logarithmic scheme $Y$ admits a strict morphism to an Artin fan~\cite[Section~3]{ACMW}. If $\mathsf A$ is an Artin fan there is an associated \textit{tropicalization}, a stack over the category of cone complexes, denoted $\Sigma$ and constructed in~\cite{CCUW}. 

\begin{customthm}{B}
The operational Chow cohomology group of $\mathsf A$ with rational coefficients is canonically isomorphic to the ring of rational piecewise polynomial functions on $\Sigma$.
\end{customthm}

If $\mathsf A$ is a global quotient of a toric variety by its dense torus, this identification was established by Payne~\cite{Pay06}. If, in addition, the stack $\mathsf A$ is smooth, the result is due to Brion~\cite{Br96}.\footnote{In fact, the results of Brion and Payne hold integrally. Note that the results of Brion and Payne are about equivariant Chow cohomology groups of toric varieties; these are, by design, the Chow groups of the quotients stacks described.}  The result equips every logarithmic scheme or stack with tautological cohomology classes from combinatorics. 

\subsection{Logarithmic pullbacks} In Section~\ref{sec: logarithmic-pullbacks} we introduce a general notion of pullback for a logarithmically flat scheme $V$ mapping to $Y$, along logarithmic local complete intersection morphism $X\to Y$ between logarithmically flat schemes. We refer the reader to main text for complete formal statements. For the moment, we note two special cases in which there is a well-behaved pullback: 
\begin{enumerate}[(i)]
\item if $X\to Y$ is a logarithmic blowup along a regularly embedded center, the pullback is the class of the strict transform;  
\item if $X\to Y$ is the diagonal morphism of a logarithmically smooth scheme, the pullback gives rise to an intersection product that is calculated by passing to a blowup of $Y$ and intersecting strict transforms.  
\end{enumerate} 

The most important part of this framework is calculating the class of the strict transform of a cycle under a blowup. With a supply of cohomology classes from piecewise polynomials, we calculate strict transforms in logarithmic geometry. Let $Y$ be a proper and logarithmically smooth scheme over a point, let $Y'\to Y$ be a proper, birational, logarithmically smooth morphism. Let 
\[
V\to Y,
\]
be a morphism from a logarithmically flat scheme. Consider a blowup $\widetilde Y\to Y$ at a regularly embedded stratum $X$ in $Y$ with exceptional $\widetilde X$. Let $[V]^{\mathsf{log}}$ be the logarithmic pullback of $V$ and $[V]^{\mathsf{sch}}$ denote the ordinary pullback along the blowup. The situation is summarized by the following diagram:
\[
\begin{tikzcd}
V^{\mathsf{log}}\arrow{r} & V^{\mathsf{sch}}\arrow{d}\arrow{r} & V\arrow{d} \\
&\widetilde Y\arrow{r} & Y.
\end{tikzcd}
\]

\begin{customthm}{C}\label{thm: strict-transform}
The difference between the classes $[V]^{\mathsf{sch}}$ and $[V]^{\mathsf{log}}$ can be calculated using the standard operations of intersection theory and the following:
\begin{enumerate}[(i)]
\item the Chow homology classes associated to strata of $V$ decorated by piecewise polynomial functions, 
\item the piecewise polynomial on $\widetilde Y$ corresponding to the excess normal bundle of $\widetilde X$. 
\end{enumerate}
\end{customthm}

\subsection{Related problems}\label{sec: related-problems} Toric contact cycles illustrate a method to understand a simple underlying geometry: natural intersection products on open moduli problems only extend correctly to sufficiently fine logarithmic compactifications, and they are related to naive extensions by blowup formulas. We list a few instances. 
\begin{enumerate}[(i)]
\item Given logarithmically smooth $X$ and $Y$, one can consider virtual classes of spaces of maps to $X$ and $Y$, and to the product $X\times Y$. After push forward to a sufficient blowup of the moduli space of curves, the three classes are related by a product formula. They are related to the corresponding classes on $\Mbar_{g,n}$ by virtual strict transforms, see~\cite{R19b}. 
\item Given a smooth projective variety $X$ and transverse, smooth divisors $D$ and $E$, one considers virtual classes of logarithmic maps to $(X,D)$, $(X,E)$, and $(X,D+E)$. In genus $0$ and in the presence of positivity conditions, the classes are again related by strict transform operations~\cite{NabijouThesis,NR19}.
\item The degeneration formula relates the virtual class of the space of stable maps on the general fiber of a simple normal crossings degeneration to the classes of stable maps to components on the central fiber. The formulation is again in terms of virtual strict transforms of products of spaces of stable maps with appropriate diagonal classes~\cite[Theorem~B]{R19}. 
\end{enumerate}

\subsection{Past and parallel developments}  Independent of logarithmic geometry, two papers have served as inspiration. The direct limit of operational Chow rings of toric blowups of a toric variety was completely described by Fulton--Sturmfels~\cite{FS97}. Aluffi's work on modification systems also played a significant role in our formulation of this problem~\cite{Alu05}. Intersection theory within logarithmic geometry has attracted significant recent attention. Herr constructs a logarithmic Gysin morphism via the logarithmic normal cone, and building on this, Barrott describes a logarithmic bivariant formalism~\cite{Bar18,Herr}; see also~\cite{BvBvG}. Their work has influenced our study, but our goals are orthogonal to theirs. Our methods explain precisely how logarithmic intersection differs from standard intersection on a model, so that existing calculations can be bootstrapped to new ones. 

We are motivated by geometric phenomena that have recently come to light, particularly by work in Abel--Jacobi theory and Gromov--Witten theory~\cite{HPS19,R19}. The ideas here are closely related to and informed by the strategy followed in~\cite{NR19}. A related direction is work by Pandharipande, Schmitt, and the first author on the top Chern class of the Hodge bundle on blowups of $\Mbar_{g,n}$~\cite{MPS20}. Piecewise polynomials appear here, but the strict transform plays no role in this work, and the \textit{formula} for the top Chern class of the Hodge bundle becomes the object of study. 

\begin{remark}
The proposal of computing logarithmic intersections via the ring of piecewise polynomials and blowup formulas was presented by D.R. at the ETH Algebraic Geometry and Moduli seminar in 2020~\cite{R20}. In the intervening time, piecewise polynomials have appeared in nearby contexts~\cite{HS21,MPS20}. The ring appears to be a natural and interesting invariant of a logarithmic scheme. 
\end{remark}

\begin{remark}
During the preparation of this paper, we were informed of work of Holmes and Schwarz, who prove that the toric contact cycles are tautological based on considerations in logarithmic Abel--Jacobi theory~\cite{HS21}. Their approach is based on new invariance properties of the logarithmic double ramification cycle. Ours is a blunt tool, and does not uncover or rely on new geometry of the cycle, but we hope this may give it broad applicability, as noted above. 
\end{remark}

\begin{remark}[Recent progress]
A preprint of the present paper first appeared in July 2021. In the intervening years, additional progress has been made in this area. By using methods from the theory of compactified Jacobians, Holmes, Pandharipande, Pixton, Schmitt, and the first author have made additional progress on toric contact cycles~\cite{HMPPS}. Precisely, they give a formula for the virtual strict transform above, i.e. the ``logarithmic double ramification cycle'' in terms of certain basic classes, including the ``piecewise polynomial functions'' studied here. A different generalization of Theorem~\ref{thm: toric-contact-cycles}, involving Brill--Noether classes, and also using compactified Jacobian techniques, is established in~\cite{Mol22}. However, the methods in the present paper are applicable more broadly. For example, the blueprint here should also show that Gromov--Witten cycles of products of logarithmically smooth curves always lie in the tautological ring, bootstrapping~\cite{FP,Jan17}. In a different direction, logarithmic Gromov--Witten classes of toric varieties with their canonical logarithmic structure, including insertions, are shown to lie in the tautological ring in~\cite{RUK22}, and intersection numbers against the logarithmic double ramification cycle are studied in~\cite{CMR22}. 
\end{remark}

\subsection*{Acknowledgements} The idea originates from discussions with Davesh Maulik concerning the degeneration formula in Gromov--Witten theory. D.R. is especially grateful to Navid Nabijou for a collaboration in which this type of study was first implemented and lessons were learned~\cite{NR19} and to Tom Graber and Rahul Pandharipande for many useful conversations about localization and the toric contact cycles in the summer of 2019. Discussions with Paolo Aluffi, Younghan Bae, Luca Battistella, Renzo Cavalieri, David Holmes, Johannes Schmitt, and Jonathan Wise have been important and inspiring. We are grateful to Dan Abramovich and Sam Payne for countless conversations that have impacted our understanding of the tropical and logarithmic geometry. We thank the referee for their comments and corrections. 

The work here was presented at the ETH Algebraic Geometry and Moduli seminar, the working seminar on bChow intersections via subdivisions, and at the Stanford Algebraic Geometry Seminar. We thank the organizers and participants for their enthusiasm and for numerous discussions.

\subsection*{Funding} S.M. was supported by ERC-2017-AdG-786580-MACI.This project has received funding from the European Research Council (ERC) under the European Union Horizon 2020 research and innovation program (grant agreement No 786580). 

\subsection*{Conventions and terminology} We work over an algebraically closed field of characteristic zero. All schemes and stacks will be of finite type unless otherwise stated. Logarithmic structures will always be fine, and sometimes saturated; in the latter case this will be explicitly stated. A logarithmic scheme $X$ is \textit{tropically smooth} if the characteristic monoid at any point of $X$ is free. Terminology such as ``logarithmically flat logarithmic scheme'' is shortened to ``logarithmically flat scheme'', etc. A \textit{logarithmic modification} $\mathcal X'\to \mathcal X$ of logarithmically smooth stacks is a proper and birational logarithmically \'etale map, which include logarithmic blowups and generalized root constructions; a logarithmic modification of a general $X$ is the pullback of such under a strict morphism $X\to \mathcal X$. A logarithmically smooth morphism is \textit{weakly semistable} if it is flat with reduced fibers and \textit{strongly semistable} or \textit{semistable} if it is weakly semistable with smooth source and target.

\section{Intersections on toric varieties}\label{sec: toric-int-thy}

We motivate the core ideas -- logarithmic flatness, the ring of piecewise polynomials, the blowup formula, and Aluffi's Segre formula -- in the toric setting. The route here is not the most efficient analysis of the toric problem, but is parallel to the moduli calculation that appears later. 

Let $Y$ be a complete toric variety and $V\hookrightarrow Y$ be an irreducible subvariety that meets the dense torus of $Y$. Let $V^\circ$ be the intersection with the torus. Suppose $f: \widetilde Y\to Y$ is a blowup of $Y$ at a regularly embedded stratum $X$.  We describe the class of the closure of $f^{-1}(V^\circ)$, or equivalently, the strict transform of $V$. The strict transform will be denoted $V^\dagger$. The upshot will be that if the class of $V$ in $Y$ is known, the strict transform can be calculated if certain ``decorated strata classes'' of $V$ are known. The latter are pushforwards to $X$ of Chow operators applied to the fundamental classes of strata of $V$. 

\subsection{Setup} A reformulation will be convenient. Fix a toric proper birational morphism $Y'\to Y$, with dense torus $T$, and choose a subvariety $V$ on $Y'$ that is transverse in the following sense:
\begin{center}
\fbox{
\textrm{the multiplication map $V\times T\to Y'$ is flat.}
}
\end{center}
Equivalently, the map $V\to [Y'/T]$ is flat. In later sections, this will become the condition that $V$ is logarithmically flat. As a consequence, if $Y'$ is blown up, the strict and total transforms of $V$ under that blowup coincide. 

\begin{remark}[platification]
Given a subvariety of a torus, its closure in any sufficiently fine toric compactification is transverse in the above sense by a theorem of Tevelev~\cite{Tev07}. 
\end{remark}

The setup is summarized by the basic blowup diagram:
\[
\begin{tikzcd}
V^\dagger\arrow{r}\arrow{d}\drar[phantom, "\square"]&\widetilde V\arrow{d}[swap]{\nu}\arrow{r}{} \drar[phantom, "\square"]& V\arrow{d}{}& \\
\widetilde Y^{\dagger '}\arrow{r} &\widetilde Y'\arrow{d}[swap]{\nu}\arrow{r}{\varphi} \drar[phantom, "\square"]& Y'\arrow{d}{\pi}& \\
 &\widetilde Y\arrow{r}[swap]{f} &Y& \fbox{$\widetilde Y = \mathsf{Bl}_XY$.}
\end{tikzcd}
\]
The map $f$ is the blowup of the regularly embedded $X$. The space $\widetilde Y^{\dagger '}$ is the strict transform of $Y'$, obtained by pulling back ideal of $X$ and blowing it up. The pullback ideal is monomial, so the blowup map 
\[
\widetilde Y^{\dagger '}\to Y'
\]
is equivariant. It is given by a subdivision of the fan of $Y'$. The remainder of the diagram is defined by pullback. 

\begin{problem}
Calculate the difference, in the Chow group of $\widetilde Y^{\dagger '}$ between the pushforward of $V^\dagger$ and the pushforward of the class $f^![V]$ to $\widetilde Y^{\dagger '}$.
\end{problem}

Since $V$ is flat over $[Y'/T]$ this is equivalent to the one stated at the beginning of the section. 

\subsection{Fulton} We use Fulton's refined blowup formula~\cite[Example~6.7.1]{Ful84}\footnote{There is a minor typographical error in the statement of this formula in Fulton's text. The stated formula is correct, but the formula is an equality in the Chow group of $\widetilde Y'$ rather than $\widetilde X'$ as stated in the line after the displayed equation. Our notation is consistent with Fulton's.}, and first apply this to the middle row of the diagram above. We remind the reader that \textit{refined} in this context takes the meaning that here that the pullback along $\widetilde Y\to Y$ is defined for an arbitrary variety equipped with a map to $Y$, rather than a cycle on $Y$, and produces a class supported on the fiber product. 


Fulton's formula involves two ingredients -- a total Segre class and a total Chern class. The Segre class term is $s(X',Y')$, where $X'$ is the pullback of the center $X$ along the toric morphism $\pi$. The Chern class term is that of the excess normal bundle $\mathbb E$ on the exceptional divisor $\widetilde X\to X$. The excess normal bundle on a blowup is the quotient of the pullback of the normal bundle $N_{X/Y}$ by the normal bundle of the exceptional divisor $\widetilde X$. We introduce notation for the exceptionals:
\[
g: \widetilde X\to X, \ \ \ \ g': \widetilde X'\to X', \ \ \ \ j: \widetilde X\hookrightarrow \widetilde Y, \ \ \ \ j': \widetilde X'\hookrightarrow \widetilde Y'.
\]

\begin{theorem}[Fulton]
There is an equality of classes in the Chow group of $\widetilde Y'$ given by
\[
f^![Y'] - [\widetilde Y^{\dagger '}] = j'_\star\left\{c(\mathbb E)\cap g'^\star s(X',Y') \right\}_{\mathsf{exp}}
\]
where the right hand side takes the expected dimensional piece of the intersection product. 
\end{theorem}

The analogue for $V$ certainly holds, and we come to it momentarily. 

\subsection{Chow} We review intersection theory on complete toric varieties. Let $Y$ be a toric variety with fan $\Sigma$ and dense torus $T$. Given a cone $\sigma$ in $\Sigma$, there is a distinguished set of linear functions defined on $\sigma$, identified with the character lattice of the dense torus. Correspondingly, there is a well-defined notion of \textit{polynomial function} from $\sigma$ to $\RR$. 

\begin{definition}
A \textit{piecewise polynomial function on $\Sigma$} is a continuous function
\[
f: |\Sigma|\to \RR
\]
whose restriction to any cones of $\Sigma$ is polynomial. Let $\mathsf{PP}(\Sigma)$ be the ring of such functions. 
\end{definition}

Results of Brion for $\Sigma$ smooth, and Payne in general, gives the ring geometric meaning~\cite{Br96,Pay06}.

\begin{theorem}[Brion/Payne]
The equivariant Chow cohomology of a toric variety $Y$ is naturally isomorphic to the ring of piecewise polynomial functions on $\Sigma$. 
\end{theorem}

The theorem holds for arbitrary toric $Y$, but the non-equivariant Chow cohomology is more delicate. An answer is known in the complete case. Let $\Sigma^{(k)}$ be the cones in $\Sigma$ of codimension $k$. 

\begin{definition}
An integer valued function $c$ on $\Sigma^{(k)}$ is \textit{balanced} if it satisfies the relation
\[
\sum_{\sigma\in \Sigma^{(k)}: \sigma\supset \tau} \langle u,n_{\sigma,\tau} \rangle \cdot c(\sigma) = 0,
\]
is satisfied for all cones $\tau$ of dimension $k+1$, and the vector $n_{\sigma,\tau}$ is the generator of the lattice of $\sigma$ relative to that of $\tau$. A balanced function of this form is a \textit{Minkowski weight} of codimension $k$. 
\end{definition}

The Minkowski weights of a fixed codimension form a group. The direct sum of the groups admits a ring structure, described in~\cite{FS97}. Fulton and Sturmfels prove the following theorem. 

\begin{theorem}[Fulton--Sturmfels]
The operational Chow cohomology ring of $Y$ is naturally isomorphic to the ring $\mathsf{MW}(\Sigma)$ of Minkowski weights on $\Sigma$.
\end{theorem}

The two theorems are connected. Katz and Payne describe the ring map
\[
\mathsf{PP}(\Sigma)\to \mathsf{MW}(\Sigma)
\]
using localization in Chow cohomology in purely combinatorial terms~\cite{KP08}. 

\subsection{Chern} For concreteness, we assume that $Y$ is \textit{smooth}. Recall the blowup diagram
\[
\begin{tikzcd}
\widetilde X\arrow{r}\arrow{d} \drar[phantom, "\square"] & \widetilde Y\arrow{d}\\
X\arrow{r}&Y.
\end{tikzcd}
\]
Let $\sigma$ be the cone corresponding to to $X$ and let $\widetilde \sigma$ be the {additional} ray in the fan of $\widetilde Y$. Enumerate the rays of $\sigma$ and consider the piecewise linear function
\[
\ell_i:\Sigma\to \RR
\]
that has slope $1$ along the $i^{\mathrm{th}}$ ray of $\sigma$ and slope zero on all other rays. The function
\[
\widetilde \ell:=\min_i \{\ell_i\}
\]
is piecewise linear on the stellar subdivision $\widetilde \Sigma$ of $\Sigma$ along $\sigma$. 

\begin{lemma}
The total Chern class of the excess normal bundle of $\widetilde X$ is the image in the Chow cohomology of $\widetilde X$ of the following piecewise polynomial class on $\widetilde Y$:
\[
c(\mathbb E) = \frac{\prod_i (1+\ell_i)}{(1+\widetilde\ell)}. 
\]
Precisely, the denominator is expanded as a formal power series, and its image in the non-equivariant Chow cohomology of $\widetilde Y$, and the Gysin pullback to $\widetilde X$ is  equal to the required Chern class.
\end{lemma}

\begin{proof}
The excess normal bundle is the quotient of the normal bundle of the center, pulled back to $\widetilde X$, by the natural relative tautological bundle on $\widetilde X$, when viewed as a projective bundle. Since the center of the blowup is the torus invariant subvariety dual to $\sigma$, and the ambient toric variety is smooth, the subvariety is a complete intersection. The normal bundle is therefore given the numerator (after interpreting each $\ell_i$ as the Chern class associated to the divisor of the $i^{th}$ ray. Similarly, by hyperplane bundle is the line bundle on $\widetilde Y$ of the exceptional $\widetilde X$, restricted to $\widetilde X$. By multiplicativity  of the total Chern class and functoriality of the identification with piecewise polynomials, the lemma follows. 
\end{proof}

\subsection{Segre}\label{sec: Segre} We have $Y'\to Y$ a proper birational map and $V\hookrightarrow Y'$ a subvariety that is flat over $[Y'/T]$.  Our goals require an understanding of the following two classes:
\[
s(X',Y') \ \ \textrm{and} \ \ s(X'\cap V, V). 
\]
The subscheme $X'$ is the vanishing locus of a monomial ideal sheaf on $Y'$ since $Y'\to Y$ is a monomial map. Similarly, the intersection $X'\cap V$ is monomial on $V$ with respect to the Cartier divisors on $V$ given by the intersections with $V$ of the toric divisors on $Y'$. 

Fix a pure scheme $S$ equipped with $n$ Cartier divisors $D_1,\ldots, D_n$ with regular crossings. They are said to have \textit{regular crossings} if the local defining equations of these divisors form regular sequences. A monomial subscheme $Z\subset S$ is determined by a finite collection of lattice points $q_1,\ldots, q_r$ in $\mathbb Z_{\geq 0}^n$. Each $q_i$ determines a reducible hypersurface supported on the $D_i$, and $Z$ is their intersection. A connected intersection of divisors contained in $Z$ will is called a \textit{stratum} in $Z$. 

A combinatorial manifestation of the normal cone is the \textit{Newton region} of $Z$. It is the complement in $\mathbb R^n_{\geq 0}$ of the convex hull of the union of the positive orthants centered at $q_1,\ldots, q_r$. The Segre classes are computed by a beautiful formula of Aluffi~\cite{Alu13}. We require the following consequence.

\begin{theorem}[Aluffi]
There exists a universal formula for the Segre class of a subscheme $Z\subset S$ that is monomial with respect to $n$ Cartier divisors with regular crossings, depending only on the Newton region of the subscheme. If $L_1,\ldots, L_n$ denote the line bundles associated to these Cartier divisors, the Segre class is equal to a sum of terms, with each term equal to a polynomial in the Chern classes of $L_1,\ldots, L_n$ applied to strata in $Z$. 
\end{theorem}

A word on the meaning of ``universal'' here. For a monomial subscheme on a regular crossings pair, the Segre class is calculated by the \textit{same} formal expression -- a sum of strata that are contained in the subscheme, decorated by Chern class operators, depending only on the Newton region. 

\subsection{Synthesis} Note that we still have the running assumption that $Y$ is smooth. We can put the pieces together now. The difference between the strict and total transforms of $V$ can be computed from the pushforwards of strata of $V$ decorated by polynomials in the Chern operators. 

Fulton's formula asserts that the difference between the strict transform and the Gysin pullback along a blowup is a class that is supported on the exceptional locus. This part of the diagram is reproduced below, with $X'$ denoting the scheme theoretic preimage of the blowup center in $Y'$:
\[
\begin{tikzcd}
\widetilde X'\arrow{r}\arrow{d}\drar[phantom, "\square"]  & X'\arrow{d}\\
\widetilde X\arrow{r} & X.
\end{tikzcd}
\]
The horizontal maps are projective bundles. The Chern class of the excess bundle $\mathbb E$ is supported on $\widetilde X$. The term $s(X',Y')$ is handled by Aluffi's formula. Recall that $V\hookrightarrow Y'$ is the cycle of interest.

\begin{lemma}
The Segre class $s(V\cap X',V)$, after pushforward to the Chow group of $X'$ is computed as
\[
s(V\cap X',V) =  s(X',Y')(V),
\]
where the right side is interpreted as follows. The Segre class $s(X',Y')$ according to Aluffi's universal formula is a sum of terms -- each term is a stratum of $X'$ times a polynomial in Chern classes of line bundles on $Y'$. Define $s(X',Y')(V)$ as the same formal sum, where the strata terms are replaced by their intersection with $V$, and the Chern classes are pulled back to these strata.  
\end{lemma}

\begin{proof}
By our transversality hypothesis for $V\hookrightarrow Y'$, the pullback of the divisors on $Y'$ exhibit $V$ as a regular crossings pair. The Segre class $s(V\cap X',V)$ can be computed by Aluffi's formula, so the lemma follows by combining the tranversality of $V$ and $Y'$ with naturality of Chern classes. 
\end{proof}

The inclusion $V\cap X'\to X'$ may not admit an obvious Gysin pullback. A more sophisticated view is that since the Segre formula is universal, it holds for monomial substacks of $[Y/T]$; since $Y$ and $V$ are flat this stack, the pullback preserves the Segre classes. We soon adopt this view, but when the definitions are unwound, this does not say anything more than the stated procedure 

We examine the blowup over the center:
\[
\begin{tikzcd}
\widetilde{V\cap X'}\arrow{r} \arrow{d}\drar[phantom, "\square"] \arrow[bend right=60,swap]{dd}{q} & V\cap X'\arrow{d} \arrow[bend left=60]{dd}{p}\\
\widetilde X'\arrow{r}\arrow{d}\drar[phantom, "\square"]  & X'\arrow{d}\\
\widetilde X\arrow{r}{h} & X.
\end{tikzcd}
\]
The composite vertical maps are proper. The vector bundle $\mathbb E$ is pulled back along $q$. By the projection formula for $c(\mathbb E)$ and compatibility of pullback and pushforward applied to the Segre class, we obtain the following equality of class in the Chow group of $\widetilde X$:
\[
q_\star\left(c(\mathbb E)\cap g'^\star s(X'\cap V,V)\right) = c(\mathbb E)\cap h^\star p_\star s(V\cap X',V).
\]
We explain the utility of the formula. The spaces in question come equipped with a collection of Cartier divisors, i.e. the boundary divisors. A \textit{standard expression} on $V\cap X'$ is a polynomial in Chern classes of the boundary divisors of $V$ applied to the class of a stratum of $V$. 

\begin{corollary}
The difference of classes
\[
f^\star[V] - [V^{\dagger}] 
\]
in the Chow group of $\widetilde Y$, is the pushforward of an element in the Chow group of $\widetilde X$. Moreover, the element can be calculated by evaluating a standard expression on $V\cap X'$, pushing forward along $p$, pulling back along $h$, and applying the total Chern class of the excess normal bundle. 
\end{corollary}

The expression above is ``universal'' in a similar sense of Aluffi's formula itself. In this toric setting, the content is that the difference between strict and total transforms can be understood completely in terms of a universal formula, whose input is the pushforward to $X$ of classes of strata of $V$ decorated by Chern classes, also known as \textit{normally decorated strata classes}~\cite{MPS20}.

\subsection{Calculus} The calculation scheme we have laid out is essentially elementary. The rest of this paper is devoted to implementing in the context of logarithmic schemes, moduli spaces, and virtual classes. We close out the toric discussion by explaining how using elementary toric intersection theory, the ingredients in the universal expression above can be pleasantly written in terms of Minkowski weights and piecewise polynomials. We stress again that $Y$ is smooth. 

\subsubsection{The cycle and its strata} The cycle $V$ in $Y'$ intersects the boundary strata of $Y'$ properly and defines a cohomology class $[V]$ in $\mathsf{CH}^\star(Y')$. It determines a Minkowski weight on $\Sigma_{Y'}$. 

Similarly, the strata of $V$ lie in the boundary strata of $Y'$. If $W\subset V$ is a closed stratum, it is contained in a closed stratum $Z'\subset Y'$. The stratum $Z'$ is a toric variety and has an associated fan. The subvariety $W$ defines a Chow cohomology class in $\mathsf{CH}^\star(Z')$. These Chow cohomology classes are expressed as Minkowski weights on the fan of $Z'$. 

\subsubsection{Normal decoration} Let $W\subset V$ be a stratum. The intersection procedure requires us to decorate this stratum with polynomials in Chern roots of the normal bundle of $W$ before pushforward. 

Assume $W\subset Z'$. The normal bundle splits, and in practice Chern roots can be given in three ways. The first is as a combination of divisorial strata of $Z'$. In this case, each of these strata is itself a toric variety and the Minkowski weight associated to $W$ described above gives a Minkowski weight on this smaller stratum. Repeating this process for all the divisors, we obtain a Minkowski weight corresponding to a decorated stratum of $V$. 

The second way that the normal roots are given is by a piecewise linear functions. In this case, we convert the piecewise linear function $\eta$ into a integer combination of boundary divisors. The coefficient of a divisor corresponding to a ray $\rho$ is the slope of $\eta$ along this ray. We then repeat the procedure above. 

The final way in which the normal roots are provided is as a product of codimension $1$ Minkowski weights. In this case, the decorated stratum of $V$ is obtained by the fan displacement rule, intersecting the cycle defined by the stratum of $V$ with these codimension $1$ weights~\cite[Section~4]{FS97}. 

\subsubsection{Pushforward} Given an equivariant proper toric morphism $Z'\to Z$ and a Minkowski weight on $Z'$, its pushforward to $Z$ is calculated via the projection formula. The Minkowski weight records the degree of the operator applied to all boundary strata. In order to compute the pushforward, it is therefore sufficient to calculate, with multiplicity, an expression in boundary strata for the pullback to $Z'$ of all boundary strata of $Z$. Once this is done, the Minkowski weight may be evaluated on these expressions, and this determines the pushforward. 

\subsubsection{Pullback} The morphism $\widetilde X\to X$ is a projective bundle and this is visible at the fan level. An explicit description of the fan of a projective bundle may be found in~\cite[Chapter VII]{CLS11} . The pullback itself is computed by Fulton and Sturmfels' formula~\cite[Proposition~3.7]{FS97}. 

\subsection{A simple example} Let us do the first nontrivial example of a strict transform calculation to illustrate the nature of the calculation.  We hope this will allow an interested reader to work through the different pieces of the formula. 

Let $Y = \mathbb P^3$ and $X$ be one of its four torus invariant points. Let $V\hookrightarrow Y$ be a line passing through $X$. We maintain the notation from the discussion above. Let us calculate the class of $V^\dagger$ in the scheme above. 

According to the scheme, in order to study the class $f^\star[V]-[V^\dagger]$, we need to calculate the excess bundle $\mathbb E$. The bundle $\mathbb E$ is a quotient of the pullback of a bundle on $X$, to $\widetilde X$. Therefore the first piece is just the normal bundle of $X$. Of course, in this case $X$ is a point, so the classes $\ell_i$ are all $0$. However, this is a low dimensional accident. 

The point $X$ is an intersection of three torus invariant hyperplanes. Each is associated to a ray, and let the associated piecewise linear functions be $\ell_1,\ell_2$ and $\ell_3$.  Since $X$ is a complete intersection, the total Chern class of the normal bundle is given by 
\[
\prod_{i=1}^3(1+\ell_i),
\]
turned into a Chow cohomology class on $Y$ and then restricted to $X$. The second piece of the excess normal bundle is the class of $\widetilde X$ on $\widetilde Y$ \textit{restricted} to $\widetilde X$. Let $\widetilde ell$ denote the piecewise linear function associated to $\widetilde X$ on $\widetilde Y$. We can expand it as a power series $1-\ell+\ell^2-\cdots$ and then restrict to $\widetilde X$. We are interested in the restriction of the linear term in this expansion, by dimension considerations. 

The final term is the Segre class of $X\cap V$ in $V$. Typically this is where we would plug in Aluffi's formula, but in this case the term is just $1$. In general, these Segre calculations are essentially independent from the Chern class calculations, so we refer the reader to the discussion in~\cite[Examples~1.1--1.4]{AluConj}. 

Putting the pieces together and extracting the degree $2$ term, we see that the difference we were trying to calculate is $-\ell$, restricted to $\widetilde X$, and then pushed forward to $\widetilde Y$. We find that there difference is exactly the class of a line in the exceptional divisor, as expected from elementary calculations.

\section{Intersections on logarithmic schemes}\label{sec: log-intersections}

In the remainder of the paper, we work with Chow groups and their operational Chow rings with rational coefficients, and assume that all logarithmic stacks have locally connected strata, to guarantee the existence of Artin fans. We write $X \times_Y^{\mathsf{log}} Z$ for the fiber product in the category of fine and saturated logarithmic schemes. 

We adapt the toric techniques to the context of logarithmic schemes. Our motivating example, the double ramification cycle, requires an additional ingredient -- the virtual class -- and we examine it in the final section.

\subsection{The plan} We examine the part of intersection theory that behaves well with respect to logarithmic blowups. The first order of business in this section is to examine intersections on logarithmically smooth schemes, and use this to model the general formalism. The relevance of Fulton's blowup formula will be apparent. 

In order to build operations in intersection theory based on logarithmic fiber products, Fulton's formula demands the following: a notion of logarithmic cycle whose logarithmic pullback is calculated by strict transforms, a supply of Chow cohomology operators on a logarithmic scheme, an understanding of the normal cones of monomial ideals, and a class of morphisms along which we can pullback. The cycles are provided by maps between logarithmically flat schemes, the cohomology operators by the ring of piecewise polynomials on the tropicalization, the normal cones by Aluffi's formula for Segre classes of monomial subschemes, and the morphisms by logarithmic local complete intersections. 

\subsection{Basic model: logarithmically nonsingular schemes} Let $Y$ be a logarithmically smooth scheme. A \textit{logarithmic cycle} on $Y$ is a pair $(V,\alpha)$ where $V$ is a logarithmically flat scheme and 
\[
\alpha: V\to Y
\]
is a proper monomorphism of logarithmic schemes. Equivalently, there exists a logarithmic modification $Y'\to Y$ such that the fine logarithmic pullback $V'\to Y'$ is a \textit{strict} closed embedding. The notion was considered in an early version of Barrott's paper~\cite{Bar18}. When the morphism $\alpha$ is clear from the context, we drop it from the notation. 

\begin{definition}
Let $(V,\alpha)$ and $(W,\beta)$ be logarithmic cycles on $Y$; their \textit{logarithmic intersection} is defined as follows. Let $\pi: Y'\to Y$ be a logarithmic modification of $Y$ with $Y'$ smooth such that, denoting the fine and saturated pullbacks of $V$ and $W$ by primes, the morphisms
\[
\alpha': V'\to Y' \ \ \ \ \beta':W'\to Y'
\]
are strict. Then define
\[
\fbox{$V\cdot_{\mathsf{log}}W: = \pi_\star(V'\cdot W')$}
\]
where the product on the right hand side is the standard intersection product on the smooth $Y'$. 
\end{definition}


\begin{proposition}
Let $(V,\alpha)$ and $(W,\beta)$ be logarithmic cycles on $Y$. The class of the {logarithmic intersection} $V\cdot_{\mathsf{log}}W$ is independent of all choices. 
\end{proposition}

\begin{proof}
By toroidal weak factorization~\cite{AKMW}, it suffices to take a blowup $Y'\to Y$ such that 
\[
\alpha': V'\to Y' \ \ \ \ \beta':W'\to Y'
\]
are flat and compare it to a further logarithmic blowup $Y''\to Y$ along a smooth center. We do this by reduction to the diagonal. The blowup gives rise to a local complete intersection morphism
\[
Y''\times Y''\to Y'\times Y'
\]
and the diagonal class on the left pushes forward to the diagonal on the right. The intersection product can be defined in two ways. First, we intersect $V''\times W''$ with the diagonal in $Y''\times Y''$ and push it forward. Second, we intersect $V'\times W'$ with the diagonal in $Y'\times Y'$. The cycle $V'\times W'$ in $Y'\times Y'$ is strict and therefore meets the strata of $X'\times X'$ in the expected dimension. Therefore the fine and saturated logarithmic pullback $V''\times W''$ has the same cycle class as the total transform. Therefore the two definitions above agree by the projection formula. 
\end{proof}


\begin{remark}[Logarithmic flatness]
Logarithmic flatness over the ground appears frequently in our setup. There are two practical reasons. First, if $Y$ is logarithmically flat, the locus where the logarithmic structure is nontrivial has positive codimension. Second, if $Y$ is logarithmically flat, a logarithmic blowup can be understood as a birational modification of $Y$. It may be possible to generalize this, but we have made no attempt to do so. In the final section of the paper, we work with a cycle, namely the double ramification cycle, that is not logarithmically flat, but we factorize out the failure of logarithmic flatness into a strict map that is handled separately, and reduce to the logarithmically flat situation. The reason this is possible is because the double ramification is still, ``virtually logarithmically flat''. The same is true for all moduli spaces of logarithmic stable maps, and the methods are easily adaptable to that setting. 
\end{remark}

\begin{remark}[Strict transforms] Fix logarithmic cycle $\alpha:V\to Y$. Logarithmic flatness implies that the Chow homology $\alpha_\star[V]$ is determined by the image of the interior of $V$, where the logarithmic structure is trivial. If $Y'\to Y$ is a projective logarithmic modification, the fine logarithmic pullback $V'$ is, by definition, obtained by pulling back the monomial ideal whose blowup defines $Y'\to Y$ to $V$ and blowing it up there. Since $V$ is logarithmically flat, this has the same Chow homology class in $Y'$ as the closure of its interior. It coincides with the class of the strict transform. 
\end{remark}

\subsection{Chow operators: piecewise polynomials}

Let $Y$ be a logarithmic scheme. There are two combinatorial constructions associated to $Y$: its cone complex\footnote{In standard terminology, a general logarithmic scheme only has a cone stack or generalized cone complex rather than a cone complex, where self-gluings {along faces} and {quotients by automorphisms} may be allowed. We will use the terminology ``cone complex'' with the allowance of the generalized situation.} $\Sigma_Y$ and its Artin fan $\mathsf{A}_Y$. The cone complex is a classical construction from toroidal geometry originating in~\cite{KKMSD} and developed further by many others~\cite{ACP,CCUW,Thu07,U13}. The cone complex has a realization within logarithmic algebraic stacks via Artin fans~\cite{AW,Ols03}; see also the survey~\cite{ACMUW}. {The output of the equivalence is that} the cone complex and the Artin fan have presentations as the same colimit over all points $y$ in $Y$ of simpler pieces:
$$
\Sigma_Y = \varinjlim \sigma_y \\
$$
and 
$$
\mathsf{A}_Y = \varinjlim \mathsf{A}_y
$$
where $\sigma_y$ are polyhedral cones, and $\mathsf{A}_y = \textup{Spec}(k[P_y])/\textup{Spec}(k[P_y^\textup{gp}])$ for $P_y$ the dual monoid of $\sigma_y$. The first colimit may be taken in generalized cone complexes~\cite{ACP} or, $2$-categorically, in the category of stacks over cone complexes~\cite{CCUW}. The latter colimit is in {logarithmic} algebraic stacks. Both colimits are taken over points of $Y$, with arrows given by generization maps. Define 
$$
\textup{PP}(\Sigma_Y) := \varprojlim \textup{PP}(\sigma_y).
$$
We pause to flag a potential point of confusion. A cone complex $\Sigma_Y$ can be presented as a colimit of cones in multiple ways; for example, a cone can be presented as a trivial colimit, or as a colimit of all of its faces. Therefore, a priori, there is a question concerning the independence of presentation of $\Sigma_Y$ in the definition, and while there is essentially a best presentation~\cite[Section~2.6]{ACP}, the issue will be dodged -- the independence will follow from the main result of this section. {Given} a presentation, we have maps 

$$
\mathsf{CH}_{\textup{op}}^\star(\mathsf{A}_Y) \to \varprojlim \mathsf{CH}_{\textup{op}}^\star(\mathsf{A}_y) \xrightarrow{\sim} \varprojlim \textup{PP}(\sigma_y) = \textup{PP}(\Sigma_Y)
$$
The middle isomorphism is due to Payne~\cite[Theorem~1]{Pay06}. When $\mathsf A_Y$ is a global quotient of a toric variety by the dense torus, Payne's result equates the outer two rings. We have the following.

\begin{theorem}
\label{theorem:Chowpp}
Let $Y$ be a logarithmic stack with Artin fan $\mathsf A_Y$. There is a natural isomorphism
$$\mathsf{CH}_{\textup{op}}^\star(\mathsf{A}_Y) = \textup{PP}(\Sigma_Y),$$
and therefore a functorial map
\[
\textup{PP}(\Sigma_Y)\to \mathsf{CH}_{\textup{op}}(Y).
\]
\end{theorem} 

\subsubsection{Proof of Theorem~\ref{theorem:Chowpp}} We begin with a result of Bae--Park that the operational Chow rings of Artin stacks (of finite type, that admit stratifications by global quotients) with rational coefficients satisfy Kimura's descent with respect to blowups\footnote{The details of the argument of Bae--Park will appear in forthcoming work~\cite{BP21}. We take the opportunity to thank them for sharing their expertise on these matters.}. A blowup $\mathcal{Y}' \to \mathcal{Y}$ induces two exact sequences: 
\[
\begin{tikzcd}
\mathsf{CH}_\star(\mathcal{Y}' \times_{\mathcal{Y}} \mathcal{Y}') \ar[r] & \mathsf{CH}_\star(\mathcal{Y}') \ar[r] & \mathsf{CH}_\star(\mathcal{Y}) \ar[r] & 0
\end{tikzcd}
\]

\[
\begin{tikzcd}
\mathsf{CH}_{\mathsf{op}}^\star (\mathcal{Y}' \times_{\mathcal{Y}} \mathcal{Y}') & \mathsf{CH}_{\mathsf{op}}^\star(\mathcal{Y}') \ar[l] & \mathsf{CH}_{\mathsf{op}}^\star (\mathcal{Y}) \ar[l] & 0 \ar[l]
\end{tikzcd}
\]
Consequently, an analogue of \cite[Theorem 3.2]{Ki92} holds, by cosmetic changes to the proof: 

\begin{lemma}
Let $p: \mathcal{Y}' \to \mathcal{Y}$ be a blowup of (finite type, that admit stratification by global quotients) Artin stacks. Let $S_i$ be the irreducible components of the locus in $\mathcal Y'$ over which $p$ is not an isomorphism, and $T_i = p^{-1}(S_i)$ their preimages.  Then a class $\alpha \in \mathsf{CH}_{\mathsf{op}}^\star(\mathcal{Y}')$ is in the image of $\mathsf{CH}_{\mathsf{op}}^\star(\mathcal{Y})$ if and only if its restriction to each $T_i$ is in the image of the Chow ring of $S_i$.    
\end{lemma} 

Let now $Y$ be a logarithmic stack, and $\mathsf{A}_Y$ its Artin fan, as in the theorem statement. There is a blowup $\mathsf{A}_{Y'} \to \mathsf{A}_X$ with $\mathsf{A}_{Y'}$ smooth. For smooth Artin fans, we have
$$
\mathsf{CH}_{\mathsf{op}}^\star(\mathsf{A}_{Y'}) = \mathsf{PP}(\Sigma_{Y'}) 
$$
by~\cite[Theorem~14]{MPS20}. Therefore, Payne's theorem and proof, on the equivariant cohomology of toric varieties, generalizes. We provide the details.

Factor $\mathsf{A}_{Y'} \to \mathsf{A}_Y$ as a sequence of blowups along smooth centers. By induction on the number of blowups necessary to reach from a non-smooth $\mathsf{A}_X$ to a smooth $\mathsf{A}_Y$, we may assume that $\mathsf{A}_{Y'} \to \mathsf{A}_Y$ is a single blowup along a smooth stratum $\mathsf{A}_S$. Let $\mathsf{A}_T$ be it's preimage, but we note that despite the notation, it is not an Artin fan but a gerbe over one. From the Kimura descent sequence, an operator $\alpha \in \mathsf{CH}_{\mathsf{op}}^\star(\mathsf{A}_{Y'})$ comes from $\mathsf{A}_Y$ if its restriction to $\mathsf{A}_T$ comes from $\mathsf{A}_S$. Let $\eta$ be the open stratum of $\mathsf{A}_S$, corresponding to a cone $\sigma \in \Sigma_Y$. The stratum may have self-intersections or automorphisms. The stabilizer group of this point is $\mathbb{G}_m^n \rtimes G$. Then $\mathbb{G}_m^n \rtimes G$ injects into the automorphism groups of every point of $\mathsf{A}_S$, and the rigidification $\mathsf{A}_S / \mathbb{G}_m^n \rtimes G$ is an Artin fan. By construction, the cone complex of $\mathsf{A}_S$ is 
$$
\mathsf{Star}(\mathsf{A}_S) : = \varinjlim_{\sigma \prec \tau \in \Sigma_Y} \tau
$$
On the other hand, let $\overline{\tau}$ denote the image of $\tau$ in the quotient of the lattice generated by $\tau$ by the lattice generated by $\sigma$, i.e. the reduced star of $\tau$, defined as 
$$
\overline{\mathsf{Star}(\mathsf{A}_S)} : = \varinjlim_{\sigma \prec \tau \in \Sigma_Y} \overline{\tau}.
$$
The cone complex $\overline{\mathsf{Star}(\mathsf{A}_S)}$ then corresponds to the cone complex of the rigidification $\mathsf{A}_S/ \mathbb{G}_m^n \rtimes G$, and, by induction on the maximal rank of stabilizer groups, we have an isomorphism 
$$
\mathsf{PP}(\overline{\mathsf{Star}(\mathsf{A}_S)}) \cong \mathsf{CH}_{\mathsf{op}}^\star(\mathsf{A}_S/\mathbb{G}_m^n \rtimes G)
$$
We are working with $\mathbb{Q}$ coefficients, so the finite group $G$ does not contribute either to piecewise polynomials or Chow rings, and since $\mathsf{A}_S$ is a gerbe over $\mathsf{A}_S/\mathbb{G}_m^n \rtimes G$, we have 
$$
\mathsf{PP}(\mathsf{Star}(\mathsf{A}_S)) = \mathsf{PP}(\overline{\mathsf{Star}(\mathsf{A}_S)}) \otimes \QQ[\mathbb{Z}^n] \cong \mathsf{CH}_{\mathsf{op}}^\star(\mathsf{A}_S/\mathbb{G}_m^n \rtimes G) \otimes \QQ[\mathbb{Z}^n] \cong \mathsf{CH}_{\mathsf{op}}^\star(\mathsf{A}_S)
$$

\noindent
Looking at the map of short exact sequences  
\[
\begin{tikzcd}
0 \ar[r] & \mathsf{CH}_{\mathsf{op}}^\star(\mathsf{A}_Y) \ar[r] \ar[d] & \mathsf{CH}_{\mathsf{op}}^\star(\mathsf{A}_{Y'}) \ar[d] \\
0 \ar[r] & \mathsf{PP}(\Sigma_Y) \ar[r] & \mathsf{PP}(\Sigma_{Y'})
\end{tikzcd}
\] 
and applying the inductive hypothesis, 
\[
\mathsf{CH}_{\mathsf{op}}^\star(\mathsf{A}_{Y'}) = \mathsf{PP}(\Sigma_{Y'}),
\] 
and the map $\mathsf{CH}_{\mathsf{op}}^\star(\mathsf{A}_Y) \to \mathsf{PP}(\Sigma_Y)$ is injective. On the other hand, let $p \in \mathsf{PP}(\Sigma_Y)$ be a piecewise polynomial. Then $p$ pulls back to a piecewise polynomial on $\Sigma_{Y'}$, i.e. a Chow class on $\mathsf{A}_{Y'}$. By Kimura's sequence, this comes from a class on $\mathsf{A}_Y$ if and only if its restriction to $\mathsf{A}_T$ comes from $\mathsf{A}_S$. But the pullback of $p$ to $\mathsf{A}_T$ comes from a class on the cone complex of $\mathsf{A}_T$, which is a subdivision of $\mathsf{Star}(\mathsf{A}_Y)$, and we've seen that piecewise polynomials on the latter are the same as $\mathsf{CH}_{\mathsf{op}}^\star(\mathsf{A}_S)$. Therefore, the pullback of $p$ to $\mathsf{A}_T$ comes from a class on $\mathsf{A}_S$, and thus $p$ is the image of this class in $\mathsf{PP}(\Sigma_X)$. Thus the map $\mathsf{CH}_{\mathsf{op}}^\star(\mathsf{A}_Y) \to \mathsf{PP}(\Sigma_Y)$ is also surjective, and hence an isomorphism. Since we have a canonical map, given by the logarithmic structure, $Y\to \mathsf A_Y$, the pullback gives rise to the claimed homomorphism to the operational Chow ring. 
\qed

\begin{remark}
The piecewise polynomial rings of smooth Artin fans can be exotic. For example, the Chow ring of the stack of expanded smooth pairs is an Artin fan, and its Chow ring has been computed by Oesinghaus, and identified with the ring of quasisymmetric functions~\cite{Oes19}, which is of significant combinatorial interest. The above result applies to the Chow ring of the stack of expansions of simple normal crossings pairs, constructed in work of Maulik and the second author~\cite{MR20}. The resulting ring is combinatorial but not well understood. 
\end{remark}

\begin{remark}[$K$-theory and cobordism]
Anderson--Payne and Gonzalez--Karu describe the equivariant operational $K$-theory and algebraic cobordism rings of toric varieties as the rings of piecewise exponential functions and piecewise graded power series functions over the Lazard ring~\cite{AP15,GK15}. The method of proof is to establish an analogue of the descent sequence. A logarithmic scheme whose Artin fan is a global quotient (which can always be achieved after a blowup) is endowed with tautologically defined classes in these theories. It is plausible that their theorems will generalize to describe the operational theories on general Artin fans. Some logarithmic aspects of $K$-theory have been explored by Chou, Herr, and Lee~\cite{CHL20}. 
\end{remark}

\begin{remark}
Holmes and Schwarz define piecewise polynomial functions on logarithmic schemes without using Artin fans~\cite[Section~3]{HS21}. The approach is based on the identification between piecewise linear functions on $\Sigma_Y$ and sections of the characteristic abelian sheaf $\overline M_Y^{\mathsf{gp}}$ of $Y$. 
\end{remark}

\subsection{Aluffi's formula and strict transforms} Let $Y$ be a logarithmic scheme with structure morphism $\epsilon: M_Y \to \mathcal{O}_Y$. A \emph{monomial subscheme} of $Y$ is a subscheme of $X$ isomorphic to the vanishing of an ideal $\epsilon(I) \subset \mathcal{O}_Y$, for $I \subset M_Y$ an ideal in the sense of monoid theory. Ideals $I \subset M_X$ are in bijection with ideals of $\overline{I} \subset \overline{M}_Y$. An ideal $I \subset M_Y$ is equivalent to the choice of a substack of the Artin fan $\mathsf{A}_Y$. If $Y$ is logarithmically flat, the map $Y \to \mathsf{A}_Y$ is faithfully flat, the ideal $I \subset \overline{M}_Y$ is determined by $\epsilon(I) \subset \mathcal{O}_Y$. We therefore have the following: 

\begin{lemma}
Let $Y$ be a logarithmically flat scheme, and $\alpha_Y: Y \to \mathsf{A}_Y$ the map to its Artin fan. The monomial subschemes of $Y$ are precisely the subschemes of the form $\alpha_Y^{-1}(T)$ for $T$ a substack of $\mathsf{A}_Y$.  
\end{lemma}

Suppose $Y$ is logarithmically flat logarithmic scheme (or stack), and $S = \alpha_Y^{-1}(T)$ a monomial subscheme of $Y$. The Segre class is preserved under flat base change~\cite[Proposition~4.2]{Ful84} and we have the following:
$$
s(S,Y) = \alpha_X^\star s(T,\mathsf{A}_Y)
$$

Monomial substacks $T\subset \mathsf A_Y$ can be both non-equidimensional and non-reduced. The operational Chow ring of $T$ is typically far from the Chow homology groups. A supply of \textit{homology} classes can be extracted using piecewise polynomials.

\begin{definition}
A homology class in $\mathsf{CH}_\star(T)$ is said to \textit{come from piecewise polynomials} if it is a linear combination of classes of the form $\pi_\star(\gamma\cap [R])$ where $\pi$ is a blowup of $\mathsf A_Y$, the class $\gamma$ is a piecewise polynomial on this blowup, and $R$ is a pure-dimensional substack of the blowup that maps to $T$. A homology class in $\mathsf{CH}_\star(\alpha_Y^{-1}(T))$ is said to \textit{come from piecewise polynomials} if it is if it the flat pullback of a homology class on $T$ that comes from piecewise polynomials.
\end{definition}

We have the following:

\begin{proposition}
The Segre class of a monomial subscheme of a logarithmically flat and tropically smooth scheme comes from piecewise polynomials. 
\end{proposition}

This lemma is existential and is an immediate consequence of a constructive theorem, namely Aluffi's formula for the Segre class of a monomial subscheme presented in~\cite{Alu13}. Aluffi states his formula in terms of \textit{regular crossings} pairs. Let $Y$ be an integral scheme and $D_1,\ldots D_s$ be divisors on $X$. They have \textit{regular crossings} if at every point $p$ in the intersection $D_{i_1} \cap \cdots \cap D_{i_j}$, the local equations for the divisors meeting $p$ form regular sequences in the local ring of $X$ at $p$. 

\begin{lemma}
A pair $(Y,D)$ is regular crossings if and only if the tautological induced map
\[
Y\to [\A^s/\mathbb G_m^s]
\]
is flat. 
\end{lemma}

\begin{proof}
The toric Artin stack $[\A^s/\mathbb G_m^s]$ with its toric boundary has regular crossings, because it has normal crossings. Regular sequences are preserved by flat pullback, so $(Y,D)$ is regular crossings. Conversely, given a regular crossings pair, since the codomain is smooth, we apply the criterion in~ \cite[\href{https://stacks.math.columbia.edu/tag/07DY}{Lemma 07DY}]{stacks-project} to conclude flatness. 
\end{proof}


\subsubsection*{Proof of Proposition via construction.}
There is a sequence of logarithmic blowups $p:Y' \to Y$ along smooth centers with $\mathsf{A}_{Y'}$ tropically smooth, monodromy free, and admits a strict morphism to $[\A^s/\mathbb G_m^s]$. The statement of the proposition is known in this case. As $Y$ is logarithmically flat and blowups are logarithmically \'etale, the space $Y'$ is also logarithmically flat and $Y' \to Y$ is birational. The birational invariance of Segre classes~\cite[Proposition~4.2]{Ful84} gives
$$
s(S,Y) = p_\star	 s(S',Y')
$$
for $S' = p^{-1}(S)$. The result follows.
\qed

We return to logarithmic intersection products. Fix $Y$ logarithmically smooth and logarithmic cycles $V_1,V_2\to Y$. In order to study the intersection product, we calculate the strict transform of a the cycle under a blowup $Y'\to Y$. It is sufficient to treat the case of a sequence of blowups of $Y$ along smooth centers, and in turn, a single blowup
\[
\widetilde Y\to Y.
\]
Theorem~\ref{thm: strict-transform} controls the logarithmic intersection product. It is restated and proved below.

\begin{theorem}
Let $V\to Y$ be a logarithmic cycle and let $\widetilde Y\to Y$ be a blowup at a smooth center $X$. Let $\widetilde X$ be the exceptional divisor. The difference between the class of the pullback $[V]^{\mathsf{sch}}$ and the class of the strict transform $[V]^{\mathsf{log}}$ is given by a universal formula involving the standard operations of intersection theory and the following:
\begin{enumerate}[(i)]
\item the Chow homology classes associated to strata of $V$ decorated by piecewise polynomial functions, 
\item the piecewise polynomial on $\widetilde Y$ corresponding to the excess normal bundle of $\widetilde X$. 
\end{enumerate}
\end{theorem}

\begin{proof}
The proof is parallel to the toric case. Let $\mathbb E$ denote the excess normal bundle of the exceptional divisor $\widetilde X$ over the center $X$ of the blowup. The difference between the strict and total transforms is determined by Fulton's formula, and is the pushforward of a class on $\widetilde X$. The Segre class of $X\cap V$ in $V$ comes from piecewise polynomials. This accounts for the first term. Since the maps $Y'\to \mathsf A_{Y'}$ and $Y\to \mathsf A_{Y}$ are smooth, the class $c(\mathbb E)$ is pulled back from $\mathsf A_{Y'}$, and therefore comes from piecewise polynomials as well. 
\end{proof}

\subsection{Interlude: Gysin pullback} We have discussed an intersection product on nonsingular logarithmic schemes based on strict transforms. We turn to logarithmic local complete intersections. 

We require generalities on Gysin pullbacks induced by maps on Artin fans. The results appear already in~\cite{Bar18, Herr}. We rework the formalism to be more transparently connected to our calculations. Fix a morphism $f:X \to Y$ between fine logarithmic schemes. By the {\it substitute for functoriality} in~\cite[Section~2.5]{AW} this map induces a commutative diagram:
\[
\begin{tikzcd}
X \ar[r, "\alpha_X"] \ar[d,"f",swap] & \mathsf{A}_{X/Y} \ar[d, "\mathsf{A}_f"] \\ Y \ar[r,"\alpha_Y",swap] & \mathsf{A}_Y 
\end{tikzcd}
\]
with strict horizontal arrows. The diagram is typically \textit{not} Cartesian so we factorize it: 
\[
\begin{tikzcd}
X \ar[r,"h"] & \mathsf{A}_{X/Y} \times_{\mathsf{A}_Y} Y \ar[r, "\alpha_Y'"] \ar[d,"\mathsf{A}_f'",swap]\drar[phantom, "\square"] & \mathsf{A}_{X/Y} \ar[d, "\mathsf{A}_f"] \\ & Y \ar[r,"\alpha_Y",swap] & \mathsf{A}_Y 
\end{tikzcd}
\]

\noindent with $h$ strict. If we further assume that the $X$ and $Y$ are tropically smooth, the stacks $\mathsf{A}_{X/Y}$ and $\mathsf{A}_Y$ are smooth, so the map $\mathsf{A}_f$ is lci. Therefore, we have a pullback 

$$
\mathsf{A}_f^{!}:\mathsf{CH}_\star(Y) \to \mathsf{CH}_\star(Y \times_{\mathsf{A}_Y} \mathsf{A}_X)
$$
\underline{Assuming} that the map $h$ is lci, we then define 
$$
h^{!}\mathsf{A}_f^{!}: \mathsf{CH}_\star(Y) \to \mathsf{CH}_\star(X)
$$

We spell out the situations of interest where $h$ is lci below. The second case will be our case of interest, but other situations may be useful to others:

\begin{description}
\item[Case I] If $Y$ is logarithmically flat over the base field, the map $\alpha_Y$ is flat. Therefore the map $\mathsf{A}_f'$, being the flat pullback of an lci map is lci. Therefore, when $h$ is lci, the map $f$ is lci, and the pullback is the ordinary lci pullback $f^!$.

\item[Case II]  If $X,Y$ are logarithmically smooth, then, being tropically smooth implies that the underlying schemes of $X,Y$ are smooth and the maps $f,h$ are automatically lci. The pullback is $f^!$, as in case $1$, without further assumptions. This is the main case of interest in this paper.

\item[Case III] The map $f:X \to Y$ is logarithmically flat. Then the map $h$, though not necessarily lci, is flat, and thus a pullback $h^\star\mathsf{A}_f^!$ still exists. It yields the same pullback as $h^{!}\mathsf{A}_f^!$ if $\mathsf{A}_f$ is also lci.

\item[Case IV] The map $f$ is weakly semistable in the sense of~\cite{AK00,Mol16}. Then $\mathsf{A}_f$ is flat, and thus so is $\mathsf{A}_f'$. Thus, in this case the pullback reduces to $h^!\mathsf{A}_f'^\star$, without assumptions about $X,Y$ being tropically smooth. Note that if $f$ is weakly semistable, the relevant diagram above reduces to the simpler diagram:
\[
\begin{tikzcd}
X \ar[r, "\alpha_X"] \ar[d,"f"] & \mathsf{A}_{X} \ar[d, "\mathsf{A}_f"] \\ Y \ar[r] & \mathsf{A}_Y 
\end{tikzcd}
\]

\end{description}

\subsection{Intersection via strict transforms: logarithmic complete intersections}\label{sec: logarithmic-pullbacks}

Olsson \cite{Ols03} has explained that there is no theory of the logarithmic cotangent complex $L_{X/Y}^{\mathsf{log}}$ of a morphism of logarithmic schemes that is perfectly parallel to the classical theory. It is not possible to construct a logarithmic cotangent complex $L_{X/Y}^\mathsf{log}$ for any morphism of logarithmic schemes, which at once reduces to the ordinary cotangent complex of a strict map, is represented by the logarithmic cotangent bundle for a logarithmically smooth map, is stable under logarithmically flat maps, and has a distinguished triangle associated to a composition. Consequently, there is no ``best" possible theory of logarithmic lci maps: there is no class of maps, closed under composition and logarithmically flat pullback, in which every logarithmically smooth morphism is a logarithmically local complete intersection, and the strict logarithmic lci maps are precisely the ordinary lci maps. We work with Olsson's version. However, Olsson proves that for logarithmically flat maps, his version coincides with Gabber's and since we have this assumption throughout, the reader may choose either formalism.

We content ourselves with schemes which are \emph{logarithmically flat} over the base field. The assumption is not sharp, but simplifies the theory and covers the examples of Section~\ref{sec: related-problems} and more.

\begin{definition}
A morphism $f:X \rightarrow Y$ of fine log schemes is a \textit{logarithmic local complete intersection morphism} or \textit{log lci} for short, if it factors as $X \to P \to Y$ with $i:X \to P$ a strict regular embedding and $g:P \to Y$ logarithmically smooth.  
\end{definition}

\begin{lemma}
The log cotangent complex of a log lci morphism $f: X \to Y$ is perfect. 
\end{lemma}

\begin{proof}
Factor $f:X \to Y$ as $i: X \to P$, $g: P \to Y$. As $g$ is logarithmically flat, there is a distinguished triangle of log cotangent complexes (i.e. even in Olsson's formulation of the cotangent complex)
\[
\begin{tikzcd}
i^\star L_{P/Y}^\mathsf{log} \ar[r] & L_{X/Y}^\mathsf{log} \ar[r] & L_{X/P}^{\mathsf{log}}.
\end{tikzcd}
\]
Perfect complexes are closed under taking cones, so we conclude.
\end{proof}

\begin{lemma}
\label{lem: basechangeoflci}
Let $f:X \to Y$ be a log lci map between logarithmically flat schemes. Then, for any logarithmically \'etale map $a:Y' \to Y$, the base change $f': X' = X \times_Y^{\mathsf{log}} Y' \to Y'$ is log lci. 
\end{lemma}

\begin{proof}
Factor $f$ as $X \to P \to Y$ into a logarithmically smooth map followed by a regular closed immersion. Since logarithmically smooth morphisms pull back under arbitrary maps, it suffices to prove the statement for $f:X \to Y$ a strict regular immersion. Then, the log fiber product $X'$ coincides with the ordinary fiber product. Since $X \to Y$ is a regular immersion, its normal cone $C_{X/Y}$ is a vector bundle. As $X$ has been assumed to be logarithmically flat, the map $X' \to X$ is birational. Thus, the inclusion of the subscheme $C_{X'/Y'}$ to $a^\star C_{X/Y}$ is closed and birational, thus an isomorphism. Therefore $C_{X'/Y'}$ is a strict regular immersion, and thus $X' \to Y'$ is log lci. 
\end{proof}

We come to the main goal. Starting with a log lci map $f:X \to Y$ between logarithmically flat schemes, we construct a \emph{refined logarithmic pullback}, associating to each map $\phi: V \to Y$ from a logarithmically flat scheme $V$ to $Y$ a class in $\mathsf{CH}_\star(X\times_Y^{\mathsf{log}}V)$. We first discuss the ordinary pullbacks for $f$. 

From the previous subsection, there is a pullback $\mathsf{CH}_\star(Y) \to \mathsf{CH}_\star(P)$. Composing with the pullback $i^{!}: \mathsf{CH}_\star(P) \to \mathsf{CH}_\star(X)$, we get a pullback $\mathsf{CH}_\star(Y) \to \mathsf{CH}_\star(X)$. We do not know when this pullback is independent of the choice of factorization $f = g \circ i$. However, if we impose additional constraints, the pullback is independent of those choices.

\begin{lemma}
\label{lem: Gysinpullback}
Suppose $f:X \to Y$ is a log lci morphism. Assume that either 
\begin{enumerate}[(i)]
    \item $X,Y$ are logarithmically flat and tropically smooth. 
    \item $Y$ is logarithmically smooth and tropically smooth, i.e. smooth. 
    \item $f$ can be factored as $i: X \to P$, $g:P \to Y$ with $g$ logarithmically smooth and weakly semistable. 
\end{enumerate}
Then the pullback is independent of the choices made. 
\end{lemma}   

\begin{proof}
We argue that when the hypotheses are met, the morphism $g: P \to Y$ is more well behaved than an arbitrary logarithmically smooth map\footnote{The reader is warned that logarithmically smooth maps can have non-equidumensional and non-reduced fibers.} $P \to Y$; if (i) is satisfied, factor $f$ as $i:X \to P',g': P \to Y$ with $g'$ log smooth. Removing all strata of $P'$ which are disjoint from the image of $X$ produces an open $P \subset P'$ with a factorization $i:X \to P, g:P \to Y$ with $P$ tropically smooth as well, and thus $f = g \circ i$ is then an ordinary local complete intersection by Cases I and III of the analysis of the previous subsection; if (ii) is satisfied, it is a map to a smooth space and possesses a pullback associated to the diagonal of this space; if (iii) is satisfied the map $g$ is flat, by Case IV of the previous subsection. Furthermore, the pullback $\mathsf{CH}_\star(Y) \to \mathsf{CH}_\star(P)$ defined via the relative Artin fan is simply the Gysin pullback or flat pullback $g^{!}$ \footnote{We have used the notation $g^!$ here for both flat and local complete intersection pullback.}.  In each of these cases, by \cite[Example~18.3.17]{Ful84}, $f$ defines a bivariant class $f^{!}$ in $\mathsf{CH}_{\mathsf{op}}^\star(X \to Y)$, and we have $f^{!} = i^! \circ g^!$. Thus, as long as one of the hypotheses are satisfied, the pullback depends only on $f$ and not the particular choice of factorization meeting these hypotheses. 
\end{proof}

We may unambiguously denote the pullback in those cases by $f^!$. We want to emphasize here that we \underline{have not} defined a Gysin pullback for the underlying map of schemes of a log lci morphism that in any way generalizes the operations in~\cite{Ful84}. We have only defined a pullback when $f: X \to Y$ can be factored as $X \to P \to Y$ of a very particular form, precisely because the constructions of \cite{Ful84} then apply. An arbitrary log lci map does not necessarily have a factorization $X \to P \to Y$ satisfying the additional assumptions of Lemma~\ref{lem: Gysinpullback}. In order to achieve such a factorization, the map $f$ may have to be modified.

\begin{definition}
Let $f: X \rightarrow Y$ be a map of logarithmic schemes. A \textit{logarithmic modification} of $f$ is a map $f':X' \rightarrow Y'$ where $Y'$ is a logarithmic modification of $Y$ and $X'$ a logarithmic modification of $X$, and $f'$ factors through the fine and saturated logarithmic base change of $f$ along $Y'\to Y$.
\end{definition}

Suppose now that $f:X \to Y$ is log lci, $X,Y$ logarithmically flat. We factor $f$ as $i: X \to P$, $g: P \to Y$ with $i$ a strict regular embedding, $g$ logarithmically smooth. We can perform strong semistable reduction, and there are two ways to proceed. We may apply difficult but complete semistable reduction in~\cite{ALT18}, or use non-representable morphisms. In the latter case, we first perform universal \textit{weak} semistable reduction~\cite{Mol16}, pass to barycentric subdivisions, then root, to obtain a non-representable semistable morphism of Deligne--Mumford stacks. 

In any event, we obtain a modification $g': P' \to Y'$ with $g'$ logarithmically smooth, weakly semistable, and $P',Y'$ tropically smooth. Pull back $X$ via $P' \to P$ to get a strict regular embedding 
\[
X' = X \times_P P' = X \times^\mathsf{log}_P P' \to P',
\]
by Lemma \ref{lem: basechangeoflci}. For a log lci map $f: X \to Y$, we can find a modification $f': X' \to Y'$ which is log lci and semistable, i.e. log lci, weakly semistable, and with $X',Y'$ tropically smooth. Furthermore, $Y'$ can be chosen to refine any particular log modification $\tilde{Y} \to Y$: simply, the pullback $\tilde{f}: \tilde{Y} \times^\mathsf{log}_Y X \to \tilde{Y}$ remains log lci, and performing this construction to $\tilde{f}$ produces such a $Y'$. In total, starting from 
\begin{enumerate}[(i)]
    \item A log lci map $f: X \to Y$ between logarithmically flat log schemes. 
    \item A log map $\phi: V \to Y$ from a logarithmically flat scheme $V$. 
\end{enumerate}
we find a modification $f': X' \to Y'$ of $f$ with $f'$ log lci and semistable, and for which the map $\phi': V' = V \times^\mathsf{log}_Y Y' \to Y'$ is weakly semistable. 

We use these observations to construct the logarithmic Gysin map. 

\begin{definition}[Logarithmic intersection]
Let $f:X \rightarrow Y$ be a logarithmic local complete intersection morphism, and $\phi: V \rightarrow Y$ be a logarithmic morphism. Assume $X,Y,V$ are logarithmically flat over the base field. Let 
\[
\begin{tikzcd}
X' \ar[r,"f'"] \ar[d,"p"] & Y' \ar[d,"q"] \\ X \ar[r,"f"] & Y 
\end{tikzcd}
\]
be a log modification of $f$ with $X',Y'$ tropically smooth, and $f'$, and $\phi':V' \to Y'$ weakly semistable, where $V' = V \times^\mathsf{log}_Y Y'$ is the fiber product in the category of fine logarithmic schemes. Define 
\[
f_{\mathsf{log}}^{!}(V)  =(p)_\star(f')^{!}(V') \in \mathsf{CH}_\star(V \times^{\mathsf{log}}_{Y} X)
\]
\end{definition}

We prove that the cycle is well-defined. 

\begin{theorem}
The cycle is obtained by logarithmic refined intersection is independent of choices. 
\end{theorem}

\begin{proof}
We show that if $f: X \to Y$ is log lci and semistable, $\phi: V \to Y$ weakly semistable, and 
\[
\begin{tikzcd}
X'' \ar[r,"f''"] \ar[d,"p"] & Y'' \ar[d,"q"] \\ X \ar[r,"f"] & Y 
\end{tikzcd}
\]
a second diagram, as in the paragraph before the theorem, with $f''$, $\phi''$ having the same properties, then $(f)^{!}(V)  =(p)_\star(f'')^{!}(V')$. Since $\phi$ is weakly semistable, $V''= V \times_Y Y''$ is Caresian in the category of schemes as well, and since $V$ is logarithmically flat, we can replace $V$ with the closure of its interior without changing it class. The map $V'' \to V$ is birational on every irreducible component. As $X,Y,X'',Y''$ are tropically smooth, the maps $p,q$ are lci, and we have $q^{!}(V) = V''$, as $q^{!}(V)$ is a $\dim V''$ dimensional class supported on $V''$. On the other hand, we have $f''^{!}q^{!}(V) = p^{!}f^{!}(V)$. Therefore, $p_{\star}f''^{!}(V') = p_{\star}p^{!}f^{!}(V)$. As $p: X'' \to X$ is proper and birational, we have $p_{\star}p^{!} = \mathsf{id}$, and so the result follows.   
\end{proof}

We end the section with a series on remarks on practicalities of this definition. 

\begin{remark}[Complexity of the operation]
The most significant combinatorial complexity comes from finding an appropriate logarithmic modification. This is a problem in polyhedral geometry. Once done, the obstruction theory on $f'$ can be broken into a regular embedding and a weakly semistable morphism, which can be understood separately. 
\end{remark}

\begin{remark}[Complexity of semistabilization]
In practice, semistabilizing the map $X \to Y$ can be difficult. Even when the simpler non-representable strong semistabilization is used, the barycentric subdivision that is used in the construction adds enormous combinatorial complexity to the problem. When $X \to Y$ is logarithmically a monomorphism, and so $\Sigma_X \to \Sigma_Y$ is injective on cones, semistabilizing $X \to Y$ is equivalent to weakly semistabilizing it, which is a much simpler process.
\end{remark}

\begin{remark}[Key example]
The most important logarithmic Gysin map for us is when the map $X \to Y$ is a logarithmic modification, with $X,Y$ smooth and logarithmically smooth. Spelling out the construction given $f:X \to Y$: form the diagram 

\[
\begin{tikzcd}
X \ar[r] \ar[d] & X \ar[d] \\ X \ar[r] & Y
\end{tikzcd}
\]
In this case, $f_{\mathsf{log}}^{!}(V) = X \times^\mathsf{log}_Y V$ is {the class of the strict transform of $V$}. The construction depends on the choice of map $\phi:V \to Y$, and does not pass to rational equivalence. The construction is expected to pass to any reasonable notion of strict rational equivalence, which should include tropical homotopies, but we have not pursued this.  
\end{remark}

\subsection{Calculation strategy}\label{sec: calculation-strategy}
The following situation is common, and paradigmatic\footnote{This subsection was written in response to a suggestion of David Holmes to emphasize the importance of the refined version of Fulton's Gysin pullback formalism. We take the opportunity to thank him for these comments.}. Consider 
\[
\begin{tikzcd}
  & \mathcal{Z} \ar[d] \\ X \ar[r] & Y
\end{tikzcd}
\]
\noindent where $\mathcal{Z}$ and $X$ are logarithmic modifications of a logarithmically and tropically smooth $Y$. The three spaces are birational, but as we exhibit, there is a subtle \textit{refined} class on the fiber product.

We form the fiber product either to the diagram as given, or after weakly semistabilizing the maps as above. Up to further birational modifications, the latter coincides with the fine and saturated logarithmic fiber product $\mathcal{Z} \times_Y^{\mathsf{log}} X$. This latter fiber product birational to $X$, but it may only be one component in the schematic pullback $\mathcal{Z} \times_Y X$; the ordinary pullback posseses \textit{additional components} of expected or even excess dimension. By Fulton's theory, these excess components support a class in expected dimension, namely the dimension of any of $X$, $Y$, or $\mathcal Z$. 

The refined pullback of the fundamental class $[\mathcal Z]$ is a class on $\mathcal{Z} \times_Y X$. It \textit{differs} from the logarithmic pullback, which the class of $\mathcal{Z} \times_Y^{\mathsf{log}} X$. These spurious classes are supported on the excess components. The following examples illustrate these phenomena.

\begin{example}
Let $Y$ be $\A^2$ and $X$ be the blowup at the origin with exceptional divisor $E$. Let $Z$ be the open subset of $X$ obtained by removing the two torus fixed points in $E$; it is abstractly isomorphic to $\A^1\times \mathbb G_m$. Everything is endowed with its toric logarithmic structure. The logarithmic fiber product is a toric variety, and in particular irreducible. The schematic fiber product contains a copy of $\mathbb P^1\times \mathbb G_m$ that maps to the origin in $Y$. Since everything has the expected dimension, the spurious component contributes positively to the intersection product~\cite[Chapter~7]{Ful84}. 

Components can have excess dimension: if $Y$ is $\A^r$, with $X$ its blowup at the origin, and $\mathcal Z$ the inclusion of $\A^1\times \mathbb G_m^{r-1}$, then the schematic pullback has a component of dimension $2r-2$.
\end{example}

The passage to the ``real world'' situation in the next section requires replacing $\mathcal Z\to Y$ with a proper morphism. In favourable situations, we are routinely led to consider diagrams of the form  
\[
\begin{tikzcd}
X \times_Y^{\mathsf{log}} Z \ar[r] \ar[d,"q"] & X \times_Y Z \ar[r] \ar[d,"r"] & Z \ar[d,"p"] \\ \ar[r] X \times_Y^{\mathsf{log}} \mathcal{Z} & X \times_Y \mathcal{Z} \ar[r] \ar[d] & \mathcal{Z} \ar[d] \\& X \ar[r,"f"] & Y
\end{tikzcd}
\]
with the upper squares Cartesian, $p$, $q$ and $r$, carrying compatible obstruction theories. We are interested in the pushforward to $X$ of the difference
$$
\boxed{r^!f^!(\mathcal{Z}) - q^!f_{\mathsf{log}}^!(\mathcal{Z})} 
$$
which amounts to applying the virtual intersection formalism to piecewise polynomials on the excess components of $X \times_Y \mathcal{Z}$. {While the treatment of the class $f^!(\mathcal{Z}) - f_{\mathsf{log}}^!(\mathcal{Z})$ is totally uniform and always produces a piecewise polynomial class, the interaction of this with the virtual class is more delicate, and depends on the obstruction theory. In the case of the double ramification cycle, this interaction is dealt with by means of the gluing formula.}

\noindent 

\section{Toric contact cycles}

The techniques will now be used study \textit{toric contact cycles}, and the remainder of this paper sets up and executes a calculation scheme for these cycles. 

\subsection{The plan} The toric contact cycles, recalled below, parameterize curves that admit a prescribed map to a fixed toric variety. In $\cM_{g,n}$ these are intersection products of double ramification cycles. The first steps in this section describe how this product relation extends to sufficiently refined blowups of $\Mbar_{g,n}$, see Section~\ref{sec: product-rule}. On any blowup, there is a lifted double ramification cycle that pushes down to the ordinary double ramification cycle on $\Mbar_{g,n}$, explained in Section~\ref{sec: lightning-class}. The products of the lifted classes recover toric contact cycles. The lifted cycle is identified with a {virtual strict transform} of the double ramification cycle.\footnote{The construction of a lifted class is not original to this paper. It was first considered by Holmes, Pixton, and Schmitt~\cite{HPS19}. Our perspective builds on its Gromov--Witten theory incarnation~\cite{R19,R19b}. The contribution here is its understanding as a virtual strict transform.} We explain the idea behind this concept. 

There is a moduli space $\DR(\sV)$ that is equipped with a morphism
\[
\DR(\sV)\to \Mbar_{g,n},
\]
which carries a virtual class in homological degree $2g-3+n$. Its pushforward to $\Mbar_{g,n}$ is well-understood~\cite{FP,JPPZ}. The morphism may be factored as:
\[
\DR(\sV)\to t\DR(\sV)\to \Mbar_{g,n}
\]
where the second morphism is an ``open'' logarithmic modification, i.e. an open in a logarithmic modification, and the first arrow is strict and virtually smooth of expected codimension $g$. The first arrow contains the virtual structure but it is strict, and so has no tropical complexity. The second arrow is far from strict, but is birational and has no virtual complexity. 

We blowup $\Mbar_{g,n}$ and apply Fulton's \underline{refined} blowup formula~\cite[Example~6.7.1]{Ful84} to the fundamental class of $t\DR(\sV)$. As in the examples of Section~\ref{sec: calculation-strategy}, we obtain a class measuring the difference between the strict and total transform of $[t\DR(\sV)]$. Virtual pullback along the first arrow in the factorization produces an expression for the virtual pullback of the strict transform, i.e. the virtual strict transform. This is the heart of the analysis, and takes place  starting in Section~\ref{sec: virtual-strict-transforms}.

\subsection{Nonsingular and compact type curves}\label{sec: nonsingular-curves} We begin in earnest. Let $\sV_1,\ldots, \sV_r$ each be vectors in $\ZZ^n$ whose components have vanishing sum, referred to as \textit{contact vectors}. Denote the ordered collection as $\underline \sV$. Each $\sV_i$ determines contact (or ramification) data for a map from $n$-pointed curves to $\mathbb P^1$: the $j^{\textrm{th}}$ entry of $\sV_i$ is the zero order of the map at the $j^{\textrm{th}}$ marking, where a negative zero order is interpreted as a pole order. The list $\underline \sV$ determines contact data for a map from an $n$-pointed curve to $(\mathbb P^1)^r$. 

 Let $\TC^\circ(\underline \sV)$ be stack of maps from smooth $n$-pointed curves of genus $g$ to the toric pair
\[
(C,p_1,\ldots, p_n)\to (\mathbb P^1)^r,
\]
with ramification specified by $\underline \sV$ and considered up to the torus action on the codomain. 

\begin{proposition}
The map $\TC^\circ(\underline \sV)\to\cM_{g,n}$ is a closed immersion.
\end{proposition}

\begin{proof}
On a fixed pointed curve $(C,p_1,\ldots, p_n)$, if such a map exists it is unique up to the rank $r$ torus scaling accounted for above. Indeed, it is equivalent to the condition that each of the Cartier divisors $\sV_i\cdot \underline p$ is trivial in the Picard group of $C$. If the $r$ line bundles are trivial, each has a unique section up to scaling, giving the morphism to $\mathbb P^1$ on each factor. 
\end{proof}

The construction therefore determines a cycle (and not merely a cycle class) in $\cM_{g,n}$ whose expected codimension is $rg$. There is a basic intersection relation:
\[
\TC^\circ(\underline \sV) = \bigcap_{i=1}^r \DR^\circ(\sV_i).
\]
This relation is visible already at the level of \textit{cycles}, rather than merely cycle classes. Note that the cycle class is $0$ in Chow homology, because the tautological ring is $0$ in codimension $g$. The product relation extends to the moduli space of curves of compact type, and passing to Chow homology, gives a basic product relation. The extension to $\Mbar_{g,n}$ is much more subtle, and is essentially our main topic in this part of the paper. We require tropical techniques to go further.

\subsection{A refresher on tropical moduli} Tropical geometry provides a route to a compactification of this locus. The discussion here is contained in~\cite{MW17} but we rework it into a setup that is more tailored to the double ramification cycle. 

The moduli space $\Mbar_{g,n}$ of stable genus $g$ curves with $n$ markings carries a logarithmic structure from the divisor of singular curves. It has an Artin fan $\mathsf A_{g,n}$ and a strict morphism
\[
\Mbar_{g,n}\to \mathsf A_{g,n}. 
\]
The Artin fan is the colimit over points $x$ in $\Mbar_{g,n}$ of stacks $\mathsf{A}_x = \textup{Spec}(k[P_x])/\textup{Spec}(k[P_x^\mathsf{gp}])$ where $P_x$ is the characteristic monoid at $x$. The colimit is taken in logarithmic algebraic stacks.

It is convenient to understand the stack $ \mathsf A_{g,n}$ in more combinatorial terms. A notion of \textit{stack over the category of cone complexes} has been introduced in~\cite{CCUW} leading to a categorical treatment of tropical moduli problems, and the moduli space of curves. 

To each $n$-pointed curve $(C,p_1,\ldots,p_n)$ one may associate  {\em weighted dual graph}: 
$$\mathsf{G} =\ (V, E, L, h),$$

\begin{enumerate}[(i)]
\item the vertex set $V$ is the set of components of $C$;
\item the edge set  $E$ is the set of nodes of $C$, where an edge $e\in E$ is incident to vertices $v_1,v_2$ if the corresponding node lies on both corresponding components;
\item the { ordered} set of { legs} of  $L$ labelled $\{1,\ldots, n\}$ with a marked leg incident to the vertex whose associated component contains the marking;
\item the function $h: V \to \ZZ_{\geq 0}$ is the genus function, where $h(v)$ is the geometric genus of the component corresponding to $v$.
\end{enumerate}

The \textit{genus} of a weighted dual graph is the sum of the values of the genus function at all vertices, plus the first Betti number of the geometric realization of $\mathsf G$. A \textit{tropical curve} is a pair $(\mathsf G,\ell)$ where $\mathsf G$ is a weighted dual graph as above and 
\[
\ell: E\to \RR_{>0}
\]
is a length function. The function $\ell$ enhances the topological realization of the graph to a metric space. We enhance further this by attaching, for each leg, a copy of $\RR_{\geq 0}$ to this metric where $0$ is identified with the vertex which supports the leg. 

If a weighted dual graph $\mathsf G$ is fixed, the interior of the cone $\sigma_{\mathsf G}:=\RR_{\geq 0}^E$ is the moduli space of all tropical curves equipped with an identification of the underlying weighted dual graph with $\mathsf G$. We refer to it as the \textit{moduli cone of $\mathsf G$}. A \textit{graph contraction} of a weighted dual graph is a sequence of edge contractions $\pi: \mathsf G\to \mathsf G'$ of the underlying graph with its canonically defined genus function given by assigning a vertex $v'$ the genus of $\pi^{-1}(v')$. An edge contraction $\mathsf G\to \mathsf G'$ induces a morphism
\[
\sigma_{\mathsf G'}\hookrightarrow \sigma_{\mathsf G}
\]
identifying the former with a face of the latter, obtained by setting the length function on coordinates corresponding to contracted edges to be zero. 

The moduli space of tropical curves may now be constructed. There are multiple approaches: via generalized cone complexes~\cite{ACP}, via cone stacks, and via combinatorial cone stacks~\cite{CCUW}. We opt for the third, following~\cite[Section~3.4]{CCUW} which is expanded on in~\cite[Section~3.2]{UZ19}. 

Let $I_{g,n}$ be the category whose objects are weighted dual graphs of stable curves of genus $g$ with $n$-marked points. The arrows are graph contractions that are compatible with all labels. Let $\mathbf{RPC}^f$ be the category of sharp rational polyhedral cones with face morphisms. The functor
\[
\sigma: I_{g,n}\to \mathbf{RPC}^f
\]
defines a category fibered in groupoids. This is explained in~\cite[Proposition~3.2]{UZ19} together with the discussion following Definition~3.1 in the same paper. The diagram of cones associated to the functor above precisely describes the moduli stack $\cM_{g,n}^{\trop}$. 

The $2$-categories of cone stacks and Artin fans are equivalent~\cite[Theorem~3]{CCUW}. The cone stack $\cM_{g,n}^{\trop}$ defined above defines an Artin fan which we denote $a^\star\cM_{g,n}^{\trop}$. There is a morphism
\[
\trop_{g,n}:\Mbar_{g,n} \to a^\star\cM_{g,n}^{\trop}
\]
that is logarithmically smooth, strict, and surjective. 

\subsection{A subdivision of the moduli space}\label{sec: tropical-DR} The Artin fan can be used to produced an open set in birational models of $\Mbar_{g,n}$. Fix a length $n$ integer vector $\sV_i$ whose components sum to $0$. Fix a weighted dual graph $\mathsf G$ and let $\sigma_{\mathsf G}$ denote its moduli cone. As described in the paragraphs above, if a point $p$ in $\sigma^\circ_{\mathsf G}$ is fixed, we obtain a metric graph structure enhancing the realization of $\mathsf G$ as a CW complex. We denote it by $\plC_p$. 

\begin{definition}
A \textit{balanced function on $\plC_p$ with slopes $\mathsf V_i$} is a continuous function $\plC\to \RR$ such that 
\begin{enumerate}[(i)]
\item the function restricted to any edge of $\plC_p$ it is linear with integer slope,
\item the sum of the outgoing slopes at every vertex of $\plC_p$ is zero, and
\item the slope along the leg labelled $j\in \{1,\ldots n\}$ is the index $j$ entry in the contact vector $\mathsf V_i$. 
\end{enumerate}
\end{definition}

\noindent
Balancing ensures that such a function is unique up to additive constant if it exists, but it may not. 

\begin{remark}
We follow the treatment of the double ramification cycle via logarithmic geometry and curves carrying piecewise linear functions on their tropicalizations found in~\cite{MW17}. In this paper, the balancing condition is not insisted upon. We insist on it here because it makes the cone complexes that we consider somewhat more concrete. Our interest here is only in the ``usual'' double ramification cycle rather than its pluricanonical variants, which allows us to fully classify the necessary piecewise linear functions from the outset. 
\end{remark}

\begin{lemma}
The closure in $\sigma_{\mathsf G}$ of the set of points $p$ in $\sigma^\circ_{\mathsf G}$ such that the tropical curve $\plC_p$ admits a balanced function with slopes given by the contact vector $\mathsf V_i$ is a subfan of $\sigma_{\mathsf G}$, i.e. a union of cones in a subdivision.
\end{lemma}

\begin{proof}
This is the main result in~\cite[Section~4]{UZ19}. 
\end{proof}

For each cone $\sigma_{\mathsf G}$, the lemma above gives rise to a union of cones contained inside $\sigma_{\mathsf G}(\mathsf V_i)$. These cones are typically not full dimensional, and in particular, the subfan is not complete. One can see this by elementary geometry. First note that the coordinates in $\sigma_{\mathsf G}$ are the edge lengths, which we can change freely. But, if we are given a map from $\plC_p$ to $\mathbb R$, where the underlying graph is $\mathsf G$, there are constraints on moving the edge lengths while keeping the map to $\mathbb R$. Indeed, each cycle imposes one condition on the edge lengths. Therefore, one expects, and it is often the case in practice, that the subfan in the lemma has codimension equal to the number of cycles. 

The intersection of this conical subfan with a face of $\sigma_{\mathsf G}$ corresponding to a graph contraction $\mathsf G\to \mathsf G'$ is the set of tropical curves with underlying weighted graph $\mathsf G'$ that admit a balanced function with slopes $\sV_i$. It follows from the work of Ulirsch and Zakharov~\cite[Theorem~B]{UZ19} that the resulting cones glue to form a substack and is equipped with a map to the moduli stack of tropical curves 
\[
\DR(\sV_i)_{\#}^{\trop}\to \cM_{g,n}^{\trop}. 
\]
The domain is a moduli space for tropical curves with a balanced function with slopes $\sV_i$. We use $\#$ to indicate that we will momentarily change the conical and integral structure on this substack. 

\begin{remark}
Ulirsch and Zakharov construct and study the space above, as the \textit{space of principal divisors}, as a generalized cone complex. The generalized cone complex comes with a presentation as a colimit of a diagram of cones and face morphisms parameterized by an index category; the functor defines a category fibered in groupoids and gives rise to a cone stack~\cite[Section~3.2]{UZ19}. The relation between generalized cone complexes and cone stacks is handled in~\cite[Section~1.3]{U21}. 
\end{remark}

We refine the cone structure on $\DR(\sV_i)_{\#}^{\trop}$ to guarantee non-singularity of its Artin fan. In above cone structure, it may occur that, in the interior of a cone $\sigma$ in $\DR(\sV_i)_{\#}^{\trop}$ where the weighted graph is constant, the difference in the positions of two vertices may change from positive to negative depending on the chosen point in $\sigma$. We replace the existing polyhedral structure on $\DR(\sV_i)_{\#}^{\trop}$ to a finer one, avoiding this phenomenon. 

By~\cite[Section~5.5]{MW17}, there is a subdivision of $\DR(\sV_i)_{\#}^{\trop}$, possibly including passage to a sublattice\footnote{It was pointed out to the authors by Holmes that there is a minor error in the description of the rubber moduli space of maps in~\cite{MW17}, concerning the integral structure on the base of this tropical moduli space, or equivalently, a certain root construction along the logarithmic boundary; we are informed that the clarifications will appear in a revised version of~\cite{BHPSS}. The integral structure is not critical for us, so long as it ensures that the universal curve is reduced.}, where in the interior of each cone, the images of the vertices of the universal curve in $\mathbb R$ are totally ordered. Define $\DR(\sV_i)^{\trop}$ to be this polyhedral subdivision. We explain it more combinatorial terms, and record it for ease of access. 

\begin{terminology}\label{graph-terminology}
Let $p$ be a point in $\DR(\sV_i)^{\trop}$. The point $p$ determines a piecewise linear function on $\plC$ up to constant with slopes given by $\sV_i$. Choose any such piecewise linear function
\[
F: \plC_p\to \RR.
\]
\begin{enumerate}[(i)]
\item We refer to the underlying weighted graph $\mathsf G$ of $\plC_p$ as the \textbf{stable source graph at $p$}. 
\item Let $\mathcal R$ be the subdivision of $\RR$ obtained by placing a vertex at the image of each vertex of $\plC_p$. Let $T$ be the graph underlying $\mathcal R$. This graph is independent of the choice of $F$. We refer to the graph $T$ as the \textbf{target graph at $p$.}
\item Let $\plC_p'$ be obtained by subdividing $\plC_p$ at the preimages of all vertices of $\mathcal R$. Let $\Gamma$ be its underlying weighted dual graph, labeling all new vertices as genus $0$. We refer to $\Gamma$ as the \textbf{semistable source graph at $p$.}
\item The \textbf{rubber combinatorial type at $p$} is the data of the weighted dual graph $\Gamma$, the target graph $T$, the graph morphism between them, and the slope decoration for each edge and leg of $\Gamma$. 
\end{enumerate}
\end{terminology}

If the rubber type is fixed, the moduli of maps with this type is a simplicial cone.

\begin{proposition}\label{lem: simplicial} 
Fix $\sV_i$ as above and let $\mathsf G$ be a weighted dual graph. Let $\sigma_{\mathsf G}$ denote the tropical moduli cone. The closure of the set of points of $\sigma_{\mathsf G}$ that lie in $\DR(\sV_i)_{\#}^{\trop}$ and have a fixed rubber combinatorial type forms a simplicial cone. 
\end{proposition}

\begin{proof}
Fix a rubber combinatorial type and examine the graph map $\Gamma\to T$. Let $C\subset E$ be the subset of edges of $\Gamma$ that have slope $0$ and let $N$ be the set of edges in the target graph $T$. The set of points in $\DR(\sV_i)_{\#}^{\trop}$ with this rubber combinatorial type is parameterized by arbitrary positive edge lengths corresponding to the contracted edges in $\Gamma$, and arbitrary positive edge lengths corresponding to the edges in the target graph $T$. Every non-contracted edge of $\Gamma$ maps with positive slope onto one of the edges of $T$ and therefore its length is determined as well. 
\end{proof}

The space $\DR(\sV_i)^{\trop}$  comes equipped with a universal stable tropical curve $\plC\to \DR(\sV_i)^{\trop}$, obtained by pulling back the universal tropical curve from $\cM_{g,n}^{\trop}$. It is also equipped with a universal semistable tropical curve, coming from Terminology~\ref{graph-terminology}(iii). Let $\DR(\sV_i)^{\trop}$ be the cone structure refining $\DR(\sV_i)_{\#}^{\trop}$ such that two points lie in the interior of a cone if and only if their rubber combinatorial types coincide. The proposition above shows that the cones are simplicial. We choose an integral structure on $\DR(\sV_i)^{\trop}$ such that (i) the cones all in fact smooth and (ii) the universal semistable source tropical curve is weakly semistable over the tropical moduli space.

\begin{remark}
The subdivision studied by Marcus and Wise is referred to as the \textit{rubber} space for its connections to Gromov--Witten theory of non-rigid targets~\cite{GV05}. The subdivision was studied earlier in the genus $0$ context by Cavalieri, Markwig, and the second author~\cite[Section~3.1]{CMR14b} and the pictures in loc. cit. may help a reader better visualize the cones in the subdivision. A similar construction appears in~\cite[Section~2.4]{NR19}.
\end{remark}

With this conical and integral structure structure fixed, we pass to Artin fans and define
\[
t\DR(\sV): = a^\star \DR(\sV_i)^{\trop} \times_{a^\star \cM_{g,n}^\trop} \Mbar_{g,n}. 
\]
The arrow from this fiber product to $\Mbar_{g,n}$ is a logarithmically \'etale and birational morphism. There are two flat families of universal curves over $t\DR(\sV)$: (i) a universal \textit{stable curve} pulled back from $\Mbar_{g,n}$; (ii) a universal \textit{semistable curve} obtained from the universal stable curve by inserting $2$-pointed $\mathbb P^1$-components. The space is stratified by rubber combinatorial types. On the locally closed strata, the universal semistable curve is obtained from Terminology~\ref{graph-terminology}(iv). The details may be found in~\cite[Section~5.2]{MW17}. 

We proceed analogously for the toric contact cycles. Fix a vector $\underline \sV$ of $n$-tuples and construct a subcomplex
\[
\TC(\underline\sV)^{\trop}\hookrightarrow\cM_{g,n}^{\trop}.
\]
as the fiber product\footnote{There is no difficulty in defining the fiber product. The map from $\DR^{\trop}(\sV_i)$ to $\cM^{\trop}_{g,n}$ is representable by cone complexes, so we work on a cover by cone complexes, where it is defined in the standard fashion in~\cite[Section~2.2]{Mol16}.} of the all tropical spaces $\DR(\sV_i)^{\trop}$ over the inputs in $\underline \sV$, over space $\cM^{\trop}_{g,n}$. In identical fashion, the subcomplex above determines a Deligne--Mumford stack
\[
t\TC(\underline \sV)\to \Mbar_{g,n},
\]
where the map to curves is a logarithmically \'etale and birational morphism. 

\subsection{Toric contacts} Each stack $t\DR( \sV_i)$ is equipped with a universal stable curve. The birational spaces induced by the subdivisions above admit Abel--Jacobi morphisms, extending
\[
(C,p_1,\ldots,p_n)\mapsto \mathcal O_C( V_i\cdot\underline p).
\] 
The construction here is due to Marcus--Wise~\cite{MW17}, but see also~\cite{Gue16,Hol17}. The universal curve is equipped with a well-defined piecewise linear function up to constant functions, and the toric dictionary gives rise to a line bundle. This gives rise to an Abel--Jacobi section associated to $\sV_i$
\[
t\DR(\underline \sV_i) \to \mathsf{Pic}_{g,n}
\]
where the codomain is the pullback of the universal Picard group from the moduli stack of curves. Since it is a section of a smooth fibration, this Abel--Jacobi section is a regular embedding, and can be equipped with a perfect obstruction theory, and so comes equipped with a Gysin pullback~\cite[Section~4.6]{MW17}. Analogously, the factorwise morphism associated to $\underline \sV$ is
\[
t\TC(\underline \sV)\to  \mathsf{Pic}_{g,n}\times_{t\TC(\underline \sV)}\cdots \times_{t\TC(\underline \sV)}  \mathsf{Pic}_{g,n}.
\]

\begin{definition}
The \textbf{toric contact space} $\TC(\underline \sV)$ is the pullback of the factorwise zero section in $t\TC(\underline \sV)$ under the morphism above.  The virtual class, defined by Gysin pullback of the zero section is the \textbf{toric contact class}. Its pushforward to $\Mbar_{g,n}$ is the \textbf{toric contact cycle}. For a single factor, we replace $\TC$ with $\DR$ and refer to it as \textbf{double ramification}  space, class, or cycle.
\end{definition}

A word of justification is required concerning the pushforward operation in the definition. The space $\TC(\underline \sV)$ is defined by pulling back a closed subscheme of a non-proper space -- the universal Picard -- to a non-proper space. It is nevertheless proper, see~\cite[Section~4]{MW17} for the rank $1$ case. Since the diagonal is closed, the general case follows from this.

By~\cite[Section~5]{MW17} that the double ramification space defined above coincides up to saturation of the logarithmic structure with the moduli space of rubber stable maps to a non-rigid $\mathbb P^1$ with its toric logarithmic structure. We use the Gromov--Witten perspective heavily in the analysis below; the standard reference is~\cite[Section~2.4]{GV05}. A modern treatment via logarithmic geometry is given in~\cite[Section~5]{MW17}, and a careful explanation is recorded in~\cite[Section~6]{BHPSS}. 

For the higher rank case, we note that $\underline\sV$ consists of $r$ entries, the space $\TC(\sV)$ can be considered a compactification of the space of maps to $(\PP^1)^r$ up to the $\mathbb G_m^r$ rubber. \footnote{These spaces have not appeared in the literature, but appear implicitly in the boundary structure of the spaces constructed in~\cite{R19} and will appear in forthcoming work of Carocci--Nabijou.} By birational invariance, the target $(\PP^1)^r$ can be replaced with any toric compactification of the same torus~\cite{AW}. A more precise relationship can be deduced from what follows, see Remark~\ref{rem: tc-gwt}.

\subsection{Product rule}\label{sec: product-rule} For ease of bookkeeping, we assume $\underline \sV$ has two entries $\sV_1$ and $\sV_2$. We have the following basic product diagram:
\[
\begin{tikzcd}
\TC(\underline \sV)\arrow{r}{\upsilon} \arrow{d}\drar[phantom, "\square"] & \mathsf P(\underline \sV)\arrow{d}\arrow{r} \drar[phantom, "\square"]& \DR(\sV_1)\times \DR(\sV_2)\arrow{d} \\
t\TC(\underline \sV)\arrow{r} & t\mathsf P(\underline \sV)\drar[phantom, "\square"]\arrow{d}\arrow{r} & t\DR(\sV_1)\times t\DR(\sV_2)\arrow{d}{\psi} \\
&\Mbar_{g,n}\arrow[swap]{r}{\Delta} & \Mbar_{g,n}\times \Mbar_{g,n}.
\end{tikzcd}
\]
There are two deficiencies in this diagram. First, the spaces on the fundamental class of $t\TC(\underline\sV)$ does not pushforward to that of $t\mathsf P(\underline \sV)$; the map $\upsilon$ is correspondingly not virtually birational. Second, while the diagonal is equipped with a normal bundle, there is no guarantee that the arrow above it is. The issues are resolved simultaneously by weakly semistabilizing $\psi$.

Fx a complete subdivision, possibly including passage to finite index sublattices, 
\[
\cM^{\trop}_{g,\underline \sV}\to \cM_{g,n}^{\trop}
\]
with the following two properties:
\begin{enumerate}[(i)]
\item The cone stack $\cM^{\trop}_{g,\underline \sV}$ is smooth.
\item For each vector $\sV_i$ the pullback of $\DR(\sV_i)^{\trop}$ mapping $\cM^{\trop}_{g,\underline \sV}$ has the property that cones of the source surject on to cones, and the lattices of the source cones are saturated in the image. 
\end{enumerate}

We form the fine and saturated logarithmic base change to obtain
\[
\begin{tikzcd}
t\DR(\sV_i)^\dagger\arrow{d}\arrow{r} &\cM^{\trop}_{g,\underline \sV} \arrow{d}\\
t\DR(\sV_i)\arrow{r} & \cM^{\trop}_{g,n}
\end{tikzcd}
\]
The space $t\DR(\sV_i)^\dagger$ inherits an Abel--Jacobi section to its universal Picard variety, and Gysin pullback of the zero section gives rise to a space $\DR(\sV_i)^\dagger$ equipped with a virtual class. There is a new diagram
\[
\begin{tikzcd}
\TC(\underline \sV)^\dagger\arrow{r}{\upsilon} \arrow{d} \arrow{r} \drar[phantom, "\square"]& \DR(\sV_1)^\dagger\times \DR(\sV_2)^\dagger\arrow{d} \\
t\TC(\underline \sV)^\dagger \arrow[swap]{r}{\delta} \arrow{d} & t\DR(\sV_1)^\dagger\times t\DR(\sV_2)^\dagger\arrow{d}{\psi} \\
\Mbar_{g,\underline \sV}\arrow[swap]{r}{\Delta_{\underline \sV}} & \Mbar_{g,\underline \sV}\times \Mbar_{g,\underline \sV}
\end{tikzcd}
\]

\noindent
where the lower square is Cartesian in the category of fine and saturated logarithmic stacks; the upper square is Cartesian in all categories.

\begin{lemma}
The morphism $\psi$ is flat of relative dimension $0$ with reduced fibers. 
\end{lemma}

\begin{proof}
The map $\psi$ is a logarithmically \'etale with smooth target. The source is a locally toric, therefore Cohen--Macaulay. Flatness will follow from equidimensionality of the fibers. Equidimensionality and reducedness of the fibers follow from the toroidal criteria~\cite[Section~4]{AK00}; see also~\cite{Mol16}. 
\end{proof}

\begin{lemma}
The squares in the diagram are Cartestian diagrams of algebraic stacks and of fine and saturated logarithmic stacks. 
\end{lemma}

\begin{proof}
The squares are defined to be Cartesian in fine and saturated logarithmic stacks. The morphism $\psi$ is integral and saturated by construction, so the two pullbacks coincide. 
\end{proof}

\begin{theorem}\label{thm: product-formula}
The toric contact cycle in $\Mbar_{g,\underline \sV}$ is the product of the strict transforms of the double ramification cycles, i.e. there is an equality of classes in the Chow groups of $\TC(\underline \sV)^\dagger$:
\[
[\TC(\underline \sV)^\dagger]^{\mathsf{vir}} = \Delta_{\underline \sV}^![\DR(\sV_1)^\dagger\times \DR(\sV_2)^\dagger]^{\mathsf{vir}}.
\]
\end{theorem}

\begin{proof}
We examine the diagram above. The morphism $\Delta_{\underline{\sV}}$ is a regular embedding since $\Mbar_{g,\underline{\sV}}$ is smooth. By the lemma, the map $\psi$ is flat and therefore the middle row map $\delta$ is also a regular embedding. The Gysin morphisms for $\delta$ and $\Delta_{\underline{\sV}}$ coincide~\cite[Theorem~6.2(c)]{Ful84}. We will show
\[
[\TC(\underline \sV)^\dagger]^{\mathsf{vir}} = \delta^![\DR(\sV_1)^\dagger\times \DR(\sV_2)^\dagger]^{\mathsf{vir}}.
\]
The obstruction theory on the morphism 
\[
\DR(\sV_1)^\dagger\times \DR(\sV_2)^\dagger\to t\DR(\sV_1)^\dagger\times t\DR(\sV_2)^\dagger
\]
is obtained, on each factor, by pulling back the obstruction theory of the appropriate Abel--Jacobi section; the latter is a regular embedding since it is a section of a smooth fibration and therefore has a normal bundle. Similarly, the Picard group of the universal curve over $t\TC(\underline{\sV})$ is the pullback along $\delta$ of the Picard group from either factor, and Abel--Jacobi sections on $t\TC(\underline{\sV})$ obtained by composition. It follows that the obstruction theory on the vertical arrows in the top square are the same. The theorem follows by compatibility of Gysin pullback. 
\end{proof}

\begin{remark}[via Gromov--Witten theory]\label{rem: tc-gwt}
There is different route to the toric contact cycle, via Gromov--Witten theory. Fix the vector of $r$ contact vectors $\underline \sV$. Let $X$ be any proper toric variety equipped with its toric logarithmic structure. The data $\underline \sV$ determines a logarithmic stable maps space $\mathsf{K}_{g,\underline\sV}(X)$. We make a simplifying assumption that the first marked point has contact order $0$ with all divisors, that is, the first entry in each contact vector is $0$. There is a morphism 
\[
\mathsf{ev_1}: \mathsf{K}_{g,\underline \sV}(X)\to X.
\]
The target can be taken to be a product of projective lines. The Gromov--Witten cycle 
\[
\mathsf{ev_1}^{\star}([\mathsf{pt}])\cap [\mathsf{K}_{g,\underline \sV}(X)]^{\mathsf{vir}} 
\]
can be pushed forward to $\Mbar_{g,n}$. This cycle completes $\TC^\circ(\underline\sV)$ on $\cM_{g,n}$: curves that admit a map to $X$ with contact orders $\underline \sV$, up to the torus action. When $X$ is $\mathbb P^1$, this completed class coincides with the double ramification cycle. Since $\TC(\underline \sV)$ satisifes a product rule and the Gromov--Witten class also satisifes a product rule~\cite{Herr,R19b} the higher rank classes also coincide. The virtual class is pulled back under forgetting points with trivial contact order, and the pullback along $\Mbar_{g,n+1}\to\Mbar_{g,n}$ is injective on Chow groups, so there is no loss of information in the assumption. 
\end{remark}

\subsection{The lightning class}\label{sec: lightning-class} Fix a single contact vector $\sV$. Consider any \textit{complete} smooth subdivision 
\[
\cM^{\trop}_{g,\sV}\to \cM^{\trop}_{g,n}
\] 
with the property that the pullback of $\DR^{\trop}(\sV)$ is weakly semistable over $\cM_{g,n}^{\trop}$. A subdivision satisfying this property is stable under further subdivision, so this property holds for any sufficiently fine subdivision. The constructions above determine a class
\[
[\DR^\lightning(\sV)] \ \ \textrm{in} \ \ \mathsf{CH}^g(\Mbar_{g,\sV};\QQ),
\]
by pushing forward the double ramification class. We verify the independence of choices. 

\begin{proposition}
The double ramification class constructed above is stable under pullback, i.e. for any further logarithmic blowup
\[
\Mbar_{g,\sV}'\to \Mbar_{g,\sV}
\]
with smooth source the two double ramification cycles are related by pullback along the local complete intersection morphism above. 
\end{proposition}

\begin{proof}
The blowup in question is given by a subdivision of the tropical moduli stack $\cM_{g,n}^{\trop}$. By supposition, the pullback of $\DR(\sV)^{\trop}$  to $\cM_{g,\sV}^{\trop}$ is a union of cones. These cones determine an open substack $\mathcal U_{g}(\sV)$; exactly as in the previous section, the substack has an Abel--Jacobi section, and the pullback of the zero section is the proper stack which we denote $\mathsf Z_g(\sV)$ to avoid a clash of notation. Denoting the analogous constructions on $\Mbar_{g,V}'$ with primes, we have two fiber squares in the category of algebraic stacks:
\[
\begin{tikzcd}
\mathsf Z_g(\sV)'\arrow{r}\arrow{d} \drar[phantom, "\square"]& \mathsf Z_g(\sV)\arrow{d} \\
U_g(\sV)'\arrow{r}\arrow{d} \drar[phantom, "\square"]& U_g(\sV)\arrow{d} \\
\Mbar_{g,\sV}'\arrow{r} & \Mbar_{g,V}. 
\end{tikzcd}
\]
The vertical arrows in the lower square are flat and therefore it suffices to check that the double ramification cycle on $\mathsf Z_g(\sV)$ pulls back to that on $\mathsf Z_g(\sV)'$. Arguing as in the proof of Theorem~\ref{thm: product-formula}, the obstruction theories for the two vertical arrows in the upper square coincide, and the proposition follows by compatibility of Gysin pullbacks. 
\end{proof}

A basic consequence. The classes $[\DR^\lightning(\sV)]$ form a compatible system of classes on all blowups of $\Mbar_{g,n}$ and we find:
\[
[\DR^\lightning(\sV)]\ \ \textrm{in} \ \ \varprojlim \mathsf{CH}_\star(\Mbar_{g,n}')
\]
where the limit is taken over all smooth $\Mbar_{g,n}'$ obtained from $\Mbar_{g,n}$ by blowups. In fact, the class comes from one in the colimit of Chow cohomology groups under pullback. The determination of $[\DR^\lightning(\sV)]$ classes in any blowup satisfying the properties above determines the toric contact cycles.

\subsection{Projective bundle geometry}\label{sec: proj-bundle} We study $[\DR^\lightning(\sV)]$ by an inductive argument, using the self-similar structure of the boundary of the space $\DR(\sV)$ that is guaranteed by the gluing formula in relative Gromov--Witten theory: the boundary is built from smaller double ramification problems. The smaller cycles can be pushed forward to smaller moduli spaces of curves, but the induction will require a more refined pushforward -- to certain projective bundles over moduli of curves. 

The following is a rephrasing of results on the boundary structure of spaces of stable maps to rubber~\cite{FP,GV05}. Logarithmic strata of $t\DR(\sV)$ are specified by {rubber combinatorial types}
\[
\Theta = [\Gamma\to T]
\]
as in Section~\ref{sec: tropical-DR}. We remind the reader that $\Gamma$ is the semistable source graph, i.e. a weighted dual graph of a semistable curve, with edges and legs decorated by integers slopes and $T$ is a graph underlying a subdivision of $\RR$. 

Fix such a type $\Theta$ and inspect a vertex $u$ of $\Gamma$. Each flag of an edge emanating from $u$ is decorated by a slope, and we collect these slopes in a integer vector $\sV_u$ of length equal to the number of flags exiting $u$. The flags exiting $u$ are not labelled, so we pass to a finite cover to label them. The group $\mathsf{Aut}_\Theta$ of deck transformations of the cover is precisely the group of automorphisms of the graph $\Gamma$ preserving labels and commuting with the map to $T$. We denote the space $t\mathsf{DR}_{g(u)}(\sV_u)$ by $t\mathsf{DR}(u)$. There is a morphism
\[
\mathsf{cut}: t\mathsf{DR}_{[\Gamma\to T]}\to \left[\prod_{u\in V(\Gamma)} t\mathsf{DR}(u)/\mathsf{Aut}_\Theta\right]
\]
obtained by cutting the target graph, as in the degeneration/gluing formulae~\cite{CheDegForm,GV05,MR14}.

\begin{lemma}
The morphism $\mathsf{cut}$ is a torus torsor. 
\end{lemma}

The source and target of $\mathsf{cut}$ are both determined by purely tropical data, and experts in toric geometry will note that it is possible to give an essentially combinatorial proof of this lemma. We record a more geometric proof. 

\begin{proof}
We will pass to $\mathsf{Aut}_\Theta$ covers first, thereby replacing the moduli stratum $t\mathsf{DR}_{[\Gamma\to T]}$ with a cover $t\widetilde{\mathsf{DR}}_{[\Gamma\to T]}$. The cover parameterizes curves equipped a rubber combinatorial type $[\Gamma'\to T']$ \underline{together} with the data of a contraction of $\Gamma'\to T'$ to a fixed cover $[\Gamma\to T]$. 

We describe the product space first. A point in this product yields a nodal curve associated to each vertex $u$ in the graph $\Gamma$. The marked points on this curve correspond to the flags of edges in $\Gamma$ leaving $u$. If $u$ is adjacent to $u'$ via an edge $e$, we glue the curves associated to $u$ and $u'$ at the points corresponding to the flags. We have passed to the cover, so the flags may be unambiguously identified. We repeat this for all edges of $T$. By construction, the curve associated to each factor admits a unique rubber combinatorial type. The types are compatible, in the sense that if there are flags emanating from $u$ and $u'$ that together form an edge in $\Gamma$, the slopes on these flags are equal\footnote{Longtime readers of Gromov--Witten theory will notice that this is essentially the predeformability condition.}. There is a universal \textit{glued} curve 
\[
\mathcal C^{\mathsf{glued}}\to\prod_{u\in V(\Gamma)} t\mathsf{DR}(u)
\]
At each point of the moduli space, there is a well-defined rubber combinatorial type, together with a contraction to the fixed cover $[\Gamma\to T]$. Note that the glued curve is \textit{semistable}, and it may have $2$-pointed $\PP^1$ components, or ``trivial bubbles''. 

Fix an edge $e$ of $T$ and let $e_1,\ldots e_r$ be the edges of $\Gamma$ that map to $e$ in $T$, and each edge $e_i$ consists of two flags $f_{e_i,1}$ and $f_{e_i,2}$. Suppose the vertices bearing these flags are $u_1$ and $u_2$. Each flag determines a line bundle on the moduli space $t\DR(u_i)$, namely the cotangent line bundle $\Psi_{e_i,j}$ of the corresponding universal semistable curve at this marked point, where $j$ is either $1$ or $2$. On the product moduli space, we pullback and continue to denote these line bundles $\Psi_{e_i,j}$. The line bundle $\Psi^{-1}_{e_i,1}\otimes \Psi^{-1}_{e_i,2}$ is denoted $\mathcal O(e_i)$. 

We examine the difference between this product space and the covered stratum $t\widetilde{\mathsf{DR}}_{[\Gamma\to T]}$. This latter space is obtained by a modification induced by the tropical inclusion
\[
\DR(\sV)^{\trop}\hookrightarrow\cM_{g,n}^{\trop}.
\]
Fix an edge $e$ of the target graph $T$, and let $e_1,\ldots e_r$ be the edges of $\Gamma$ mapping to $e$. Let $w_i$ be the slope along $e_i$. Given a point of the product space above, there is a canonically associated glued curve with a rubber combinatorial type. However, in order to produce a point of $t\widetilde{\mathsf{DR}}_{[\Gamma\to T]}$, we must also constrain the logarithmic structure at the nodes. In order to see this, note that the weighted lengths of these corresponding edges are always equal in (the $\Theta$-cone of) $\DR(\sV)^{\trop}$. These edge lengths are piecewise linear functions on the tropical moduli space, or equivalently, elements of the characteristic sheaf of the logarithmic structure. Unwinding the manner in which $\DR(\sV)^{\trop}$ determines a birational modification of $\Mbar_{g,n}$, to produce an element of $t\widetilde{\mathsf{DR}}_{[\Gamma\to T]}$ we must choose an isomorphism between the line bundles $\mathcal O(e_i)^{\otimes w_i}$ for different $i$ whenever the elements of the characteristic are equated. 

We repeat the analysis for all edges of $T$; this exhibits the morphism
\[
\mathsf{cut}: t\widetilde{\mathsf{DR}}_{[\Gamma\to T]}\to \prod_{u\in V(\Gamma)} t\mathsf{DR}(u)
\]
as the total space of direct sum of torsors. Specifically, for each edge $e$ of $T$ with edges  $e_1,\ldots e_r$ of $\Gamma$ mapping to it, we form the total space of the torsors
\[
\mathcal O(w_1\cdot e_1-w_r\cdot e_r)\oplus\cdots\oplus \mathcal O(w_{r-1}e_{r-1}-w_r\cdot e_r).
\]
Taking another direct sum over all edges $e$ of $T$, we obtain a space isomorphic to $t\widetilde{\mathsf{DR}}_{[\Gamma\to T]}$, and this proves the lemma. 
\end{proof}

There is a map equipped with a perfect obstruction theory
\[
\DR(\sV)\to t\DR(\sV),
\]
and since the inclusion $t\mathsf{DR}_{[\Gamma\to T]}\hookrightarrow t\DR(\sV)$ is a regular embedding, the pullback of the virtual class gives a refined cycle of codimension $g$ on $t\mathsf{DR}_{[\Gamma\to T]}$ supported on a compact subset. We denote it $\mathsf{DR}_{[\Gamma\to T]}^{\mathsf{vir}}$. The induction will require pushing forward this refined class to compactifications of $t\mathsf{DR}_{[\Gamma\to T]}$. In order to induct, we relate this to a refined class on $\prod_{u\in V(\Gamma)} t\mathsf{DR}(u)$ and compactifications of it. 

The torus bundle $t\mathsf{DR}_{[\Gamma\to T]}$ contains a compact subset and class on, namely the virtual class $\mathsf{DR}_{[\Gamma\to T]}$ of the stratum. Our main argument will be to study this class on pushforwards to compactifications of $t\mathsf{DR}_{[\Gamma\to T]}$. The study will be inductive, using the fact that $t\mathsf{DR}_{[\Gamma\to T]}$ is a bundle over a product of simpler spaces. This requires a compactification of the map $\mathsf{cut}$ above.

\begin{construction}[A compactification of $\mathsf{cut}$]\label{compactify-cut}
The proof of the lemma shows that the morphism 
\[
\mathsf{cut}: t\widetilde{\mathsf{DR}}_{[\Gamma\to T]}\to \prod_{u\in V(\Gamma)} t\mathsf{DR}(u)
\]
is the total space of a direct sum of $\mathbb G_m$-torsors. Each torsor in the direct sum was obtained as a tensor product of cotangent line bundles at the marked points. Fix a marked point $\star$ and let $\mathbb L_\star$ denote the cotangent line bundle on the universal \textit{stable curve}. On the locus of non-broken curves in $t\mathsf{DR}(u)$ the line bundle $\mathbb L_\star$ coincides with the cotangent line bundle $\Psi_\star$. By the standard comparison of cotangent lines, the difference $\mathbb L_\star\otimes\Psi_\star^{-1}$ is the divisor associated to a linear combination of boundary divisors on $t\mathsf{DR}(u)$. Specifically, it is the sum of boundary divisors where the marked point $\star$ is supported on an unstable component of the universal semistable curve. 

Choose compactifications of each factor in $\prod_{u\in V(\Gamma)} t\mathsf{DR}(u)$ to form $\prod_{u\in V(\Gamma)} \Mbar(u)$. Each $t\mathsf{DR}(u)$ is birational to (a finite group quotient of) an appropriate moduli space of curves, with genus and number of marked points determined by $u$. We require the compactification to be a proper toroidal modification of the space of stable curves corresponding to these data. 

For the marked point $\star$, we choose a Cartier divisor on the compactification that restricts to $\mathbb L_\star\otimes\Psi_\star^{-1}$. Since each line bundle extends, the torus torsor also extends to the compactification $\prod_{u\in V(\Gamma)} \Mbar(u)$. Since the source of $\mathsf{cut}$ is a contained in a direct sum of line bundles that extend to the compactification in the base direction, we compactify the fiber directions of the map via the projective completion of these line bundles to obtain 
\[
\mathsf{kut}: \widetilde{\cM}_{[\Gamma\to T]}\to \prod_{u\in V(\Gamma)} \Mbar(u).
\]
By construction, this coincides with $\mathsf{cut}$ after restricting to $\prod_{u\in V(\Gamma)} t\mathsf{DR}(u)$ in the base directions and the tori in each fiber direction. 
\qed
\end{construction}

Write $[\mathsf{DR}_{[\Gamma\to T]}]^{\mathsf{push}}$ pushforward to $\widetilde{\cM}_{[\Gamma\to T]}$ of the virtual class of the stratum corresponding to $[\Gamma\to T]$. We write $\left[\prod_{u\in V(\Gamma)} \mathsf{DR}(u)\right]^{\mathsf{push}}$ for the factorwise double ramification cycle determined by each vertex. The following proposition shows that the split class can be used to recover the class of the stratum. 

\begin{proposition}\label{prop: projective-bundle-tautological}
Let $H$ denote the operational class of the fiberwise hyperplane in the bundle $\widetilde{\cM}_{[\Gamma\to T]}$. There is an equality of classes on $\widetilde{\cM}_{[\Gamma\to T]}$ given by
\[
[\mathsf{DR}_{[\Gamma\to T]}]^{\mathsf{push}} = \lambda\cdot H^k\cap \mathsf{kut}^\star\left[\prod_{u\in V(\Gamma)} \mathsf{DR}(u)\right]^{\mathsf{push}}
\]
where $\lambda$ is a nonzero rational number and $k$ is the fiber dimension of $\mathsf{kut}$. 
\end{proposition}

In the following proof,  we use results from Chen's treatment of the degeneration formula for logarithmic expanded degenerations~\cite[Section~7]{CheDegForm}. Our setting is slightly different, in that we describe the boundary of the space of maps to a rubber target, rather than a stratum in a degeneration of the moduli space. The precise results we need apply without change in our setting.

\begin{proof}
By the gluing formula in relative Gromov--Witten theory, after passing to a finite cover to kill the automorphisms, the arrow
\[
\widetilde{\mathsf{DR}}_{[\Gamma\to T]}\to \prod_{u\in V(\Gamma)} \Mbar(u),
\]
obtained by restricting $\mathsf{kut}$, is finite and \'etale over its image and its image is precisely the external product $\prod_{u\in V(\Gamma)} \mathsf{DR}(u)$. This result follows immediately from~\cite[Lemma~7.9.1]{CheDegForm}. We examine the projective bundle obtained by pulling back the bundle from Construction~\ref{compactify-cut} to this product $\prod_{u\in V(\Gamma)} \mathsf{DR}(u)$. Note that the space $\widetilde{\mathsf{DR}}_{[\Gamma\to T]}$ is naturally contained in this bundle.

We claim the bundle becomes trivial after replacing the base with a finite cover which we denote $\prod^\sim_{u\in V(\Gamma)} \mathsf{DR}(u)$. Indeed, on restricting to this locus, the morphism admits an orbifold section by~\cite[Eqn. 7.9.2]{CheDegForm}, and therefore the torsor is trivial on this locus; see also~\cite[Remark~3.4]{GV05}. The natural obstruction theory on $\prod^\sim_{u\in V(\Gamma)} \mathsf{DR}(u)$ pulls back to the obstruction theory on the section given by $\mathsf{DR}_{[\Gamma\to T]}$; this follows by the same argument presented in~\cite[Section~4.10]{CheDegForm}. Therefore, on this trivial projective bundle, the virtual class of the section is equal to the pullback of the virtual class on $\prod^\sim_{u\in V(\Gamma)} \mathsf{DR}(u)$ times a power of the fiberwise hyperplane class. The result follows from the projection formula. 
\end{proof}

\subsection{Computing virtual strict transforms}\label{sec: virtual-strict-transforms} The class $[\DR(\sV)]$ in the Chow cohomology of $\Mbar_{g,n}$ is known to be tautological, either by the result of Faber and Pandharipande~\cite{FP} or via the explicit description in terms of the standard basis of the tautological ring~\cite{JPPZ}. We compute $[\DR^\lightning(\sV)]$ on a sufficient blowup of the moduli space of curves and show that it lies in the tautological ring. 

The class $[\DR^\lightning(\sV)]$ can be regarded as a \textit{virtual strict transform} of the refined cycle $[\DR(\sV)]$ in $\Mbar_{g,n}$ in the following sense. We begin with a morphism 
\[
t\DR(\sV)\to \Mbar_{g,n}
\]
and examine its pullback under any sufficient projective birational modification $\Mbar_{g,\sV}$, which we assume to be smooth. This modification is given by the blowup of a monomial ideal sheaf, in the sense discussed previously. The pullback of $t\DR(\sV)$ is the \textit{total transform} of the morphism. 

There is a well-defined \textit{strict transform} which is obtained by pulling back the ideal sheaf above to $t\DR(\sV)$ and blowing it up there. But the ideal sheaf is monomial and therefore pulled back from the Artin fan of $\Mbar_{g,n}$, so this strict transform is obtained by first subdividing $\Mbar_{g,n}$, then pulling back the subdivision to $\DR(\sV)^{\trop}$, and performing the associated birational modification on $t\DR(\sV)$. Since the blowup is sufficiently fine, the morphism from this modification of $t\DR(\sV)$ to $\Mbar_{g,\sV}$ is the inclusion of an open substack. 

There are now three spaces in play: the space $t\DR(\sV)$, its strict transform, and its total transform. The first two are pure of dimension $3g-g+n$, while the latter is not pure, but is equipped with an excess class in homological degree $3g-3+n$ by Fulton's \underline{refined} Gysin pullback applied to the local complete intersection morphism
\[
\Mbar_{g,\sV}\to \Mbar_{g,n},
\]
see~\cite[Section~6.6]{Ful84}. The vector $\sV$ determines Abel--Jacobi sections on all three spaces, and we can perform refined intersection for each one with the $0$ section of the universal Picard group. We obtain the \textit{virtual strict transform} and \textit{virtual total transform}. By definition, after pushing forward to $\Mbar_{g,\sV}$, the virtual total transform is the class of the Chow cohomological pullback under the blowup and is known~\cite{FP,JPPZ}. We compute the virtual strict transform. 

The situation is complicated by the fact that a blowup along a monomial ideal is singular and difficult to control in intersection theory, and we instead argue inductively along a factorization. 

\noindent
{\bf Factorization.} As a first step, we require a sequence of subdivisions and root constructions
\[
\cM^{\trop}_{g,n}\langle m\rangle \to \cM^{\trop}_{g,n}\langle m-1 \rangle \to \cdots \to \cM^{\trop}_{g,n}\langle 0\rangle =\cM^{\trop}_{g,n}
\]
with the following properties: (i) each arrow is a blowup at a stack theoretically smooth center. (ii) the pullback of $\DR^{\trop}(\sV)\to \cM^{\trop}_{g,n}$ to $\cM^{\trop}_{g,n}\langle m\rangle$ is weakly semistable. 

\begin{lemma}
A sequence of blowups with the properties above exists. 
\end{lemma}

\begin{proof}
Sequentially blowup smooth strata of $\Mbar_{g,n}$ until the irreducible components of the boundary are each smooth and non-empty intersections of these irreducible strata are connected; note that this can be achieved by two barycentric subdivisions, see for instance~\cite[Section~4.6]{ACMW}. Once this is done, we use the fact that sequences of stellar subdivisions of a fan are cofinal in the system of all subdivisions. The result follows. 
\end{proof}

As our goal is to describe the difference between the virtual strict and virtual total transforms, we begin by describing this difference on the base of the obstruction theories. 

\subsection{The first blowup: subdivision} Consider the diagram
\[
\begin{tikzcd}
t\DR(\sV)\langle 1\rangle \arrow[hookrightarrow]{r}& t\DR(\sV)^{\mathsf{tot}}\arrow{d}\arrow{r} \drar[phantom, "\square"] & t\DR(\sV)\arrow{d}{\kappa}\\
&\Mbar_{g,n}\langle 1\rangle\arrow{r}{\nu}& \Mbar_{g,n}, 
\end{tikzcd}
\]
where the middle square is cartesian and defines the total transform, while $t\DR(\sV)\langle 1\rangle$ is the strict transform. Recall that the space $t\DR(\sV)$ is an open in a birational modification of $\Mbar_{g,n}$ and contains the double ramification space as a closed and proper substack (we do not yet involve the space $\DR(\sV)$ itself). In particular, aside from $t\DR(\sV)\langle 1 \rangle$, the remaining four spaces in the diagram are birational to each other. As explained in Section~\ref{sec: calculation-strategy} there is still very nontrivial refined intersection theory on the total transform!

The morphism $\nu$ is a blowup morphism and its source and target are smooth Deligne--Mumford stacks. It is therefore a local complete intersection morphism. There is a refined Gysin pullback $\nu^!$ which we apply to obtain
\[
\nu^!\left[t\DR(\sV)\right]-\left[t\DR(\sV)\langle 1\rangle \right] = [\underline{\textsf{correction}}]. 
\]
The term $[\underline{\textsf{correction}}]$ is supported on the exceptional divisor. It is described by Fulton's blowup formula in its refined version~\cite[Example~6.7.1]{Ful84}. Let $Z\langle 1\rangle$ be the center of the blowup morphism $\nu$. Denote the exceptional divisor by $\mathbb P\langle 1\rangle$. It is the projectivized normal bundle of the center, and can examine the bundle map:
\[
\mu: \mathbb P\langle 1\rangle \to Z\langle 1\rangle. 
\]
In Fulton's formula, the $[\underline{\textsf{correction}}]$ term above is pushed forward from $\mathbb P\langle 1\rangle\times_{\Mbar_{g,n}\langle1\rangle} \kappa^{-1}Z\langle 1 \rangle$ to the total transform. The class on $\mathbb P\langle 1\rangle$ is the expected dimensional, that is homological degree $3g-3+n$. Precisely, it is the degree $3g-3+n$ piece in the product of the following two terms:

\begin{description}
\item[Chern] The Chern class of the excess normal bundle -- the quotient of the normal bundle of $Z\langle 1\rangle$ by the fiberwise hyperplane in $\mathbb P\langle 1\rangle$. 
\item[Segre] The Segre class of the scheme theoretic preimage $\kappa^{-1}(Z\langle 1\rangle)$ in $t\DR(\sV)$. 
\end{description}

Let us interpret the product. The support of the Segre term is, by definition, contained in the fiber product $Z\langle 1\rangle \times_{\Mbar_{g,n}} t\DR(\sV)$. In a moment, we will explain how to handle this Segre class term using Aluffi's formula. For now, we blackbox the class. This fiber product maps to the center $Z\langle 1\rangle$ of the blowup, and therefore we can pullback the projective bundle $\mu$ to this fiber product. Now apply flat pullback to this Segre class term, along the bundle, and then apply Chern classes of the excess bundle. The result of a process is the correction term. Note that as yet, there are still no virtual class considerations. 

We come to the Segre class term and Aluffi's formula. A stratum in any logarithmic scheme is cut out by a monomial ideal sheaf. This class of ideals is stable under pullback by logarithmic morphism. The center of the blowup $Z\langle 1\rangle$ is such a stratum. It follows that the Segre class term $s(\kappa^{-1}Z\langle 1\rangle, t\DR(\sV))$ is given by a monomial ideal. We can replace $t\DR(\sV)$ with a logarithmic modification without changing the problem, and therefore we can assume that $t\DR(\sV)$ is a simple normal crossings pair, i.e. tropically smooth. 

Since the Segre class on a logarithmically and tropically smooth space comes from piecewise polynomials, we have the following.

\begin{corollary}
The \emph{\underline{\textsf{correction}}} term comes from piecewise polynomials. Explicitly, it is a sum of classes that are obtained by the following three steps:
\begin{enumerate}[(i)]
\item Choose a closed stratum $W\subset t\DR(\sV)$ mapping to the center $Z\langle 1\rangle$ of the blowup.
\item Pullback the projective bundle $\mu$ to $W$ and obtain a projective bundle $\mathbb P_W\to W$.
\item Apply to the fundamental class of the $\PP_W$ a polynomial in the Chern roots of the normal bundle $N_{Z\langle 1\rangle/\Mbar_{g,n}}$, the Chern class of the relative hyperplane bundle $\mathbb P_W\to W$, and  the Chern roots of the normal bundle $N_{W/t\DR(\sV)}$. 
\end{enumerate}
\end{corollary}


\subsection{The first blowup: virtual structure} The preceding arguments account for the purely tropical part of the argument. We now wish to calculate the pushforward of the virtual strict transform of $\DR(\sV)$ to the blowup $\Mbar_{g,n}\langle 1 \rangle$. To achieve this, we apply virtual pullback to \underline{\textsf{correction}} and push it down to $\Mbar_{g,n}\langle 1 \rangle$. We begin by adding a row to our previous diagram
\[
\begin{tikzcd}
\DR(\sV)\langle 1\rangle \arrow[hookrightarrow]{r}\arrow{d} \drar[phantom, "\square"] & \DR(\sV)^{\mathsf{tot}}\arrow{d}\arrow{r} \drar[phantom, "\square"] &\DR(\sV)\arrow{d}{}\\
t\DR(\sV)\langle 1\rangle \arrow[hookrightarrow]{r}& t\DR(\sV)^{\mathsf{tot}}\arrow{d}\arrow{r} \drar[phantom, "\square"] & t\DR(\sV)\arrow{d}{\kappa}\\
&\Mbar_{g,n}\langle 1\rangle\arrow{r}{\nu}& \Mbar_{g,n}. 
\end{tikzcd}
\]
In the middle row, the setup of the previous section gives us an equality:
\[
\nu^!\left[t\DR(\sV)\right]-\left[t\DR(\sV)\langle 1\rangle \right] = [\underline{\textsf{correction}}]. 
\]
The equality is as classes on the total space of the pullback, namely $t\DR(\sV)^{\mathsf{tot}}$.
The vertical arrows are all equipped with obstruction theories, and we pull $[\underline{\textsf{correction}}]$ up along the vertical arrows to obtain a \textit{virtual} {correction}, and push it down to $\Mbar_{g,n}\langle1\rangle$. By compatibility of Gysin maps~\cite[Proposition~6.6]{Ful84} on the top right square, the class of $\DR(\sV)\langle 1 \rangle$ can be computed by pulling back the known formula on $\Mbar_{g,n}$ and subtracting the virtual correction term. 

The term $[\underline{\textsf{correction}}]$ is supported over the blowup center. Furthermore, by Aluffi's formula, it is a sum of Chern class operators applied to strata of $t\DR(\sV)$ that map to the blowup center $Z\langle 1\rangle$. Ley us now choose a single such summand. We may therefore fix a stratum $W$ -- it amonts to choosing a cone in the tropical moduli space of Section~\ref{sec: tropical-DR}. Restricting attention here, we have:
\[
\begin{tikzcd}
\mathsf{DR}_W^{\mathsf{exc}}(\sV)\arrow{r}\arrow{d}{\epsilon} \drar[phantom, "\square"] &\mathsf{DR}_W(\sV)\arrow{d}{\varphi} \\
\mathbb P_W\langle1\rangle \arrow{d}{\tau}\arrow{r} \drar[phantom, "\square"] & W \arrow{d}{\kappa} \\
\mathbb P\langle 1 \rangle \arrow{r}{\varpi} & Z\langle 1 \rangle. 
\end{tikzcd}
\]
The space $\mathsf{DR}_W(\sV)$ is simply the preimage of $W$ in $\DR(\sV)$. The term $\mathsf{DR}_W^{\mathsf{exc}}(\sV)$ is defined by pullback. The key property is that it is a projective space bundle over the stratum $\mathsf{DR}_W(\sV)$. 

Before proceeding, we observe that there are natural classes defined in different parts of this diagram. The center $Z\langle 1 \rangle$ carries the Chern classes of its normal bundle. The exceptional divisor $\mathbb P\langle 1 \rangle$ carries the fiberwise hyperplane class. The stratum $W$ carries cohomology class that comes from piecewise polynomials, according to Aluffi's formula.

\begin{proposition}
The pushforward to $\Mbar_{g,n}\langle 1\rangle$ of the difference of classes
\[
\nu^![\DR(\sV)]-[\DR(\sV)\langle 1\rangle]
\]
is tautological. More precisely, is pushforward of a sum of terms that are supported on the exceptional divisor of $\nu$, with each given by a product of tautological classes with powers of the normal bundle of the exceptional divisor. 
\end{proposition}

\begin{proof}
We apply Fulton's blowup formula to $t\DR(\sV)$ as stated in~\cite[Example~6.7.1]{Ful84} to obtain the term $[\underline{\textsf{correction}}]$. We perform virtual pullback and pushforward of this class to control it. By Aluffi's Segre class formula, we see that the classes of interest in Fulton's formula are sums over strata $W$ in $t\DR(\sV)$ of classes are obtained in the following steps:
\begin{enumerate}[(i)]
\item Choose any piecewise polynomial $p$ on $t\DR(\sV)$ and pull it back to $W$;
\item Choose a polynomial $c$ in the Chern classes of the normal bundle to $Z\langle 1 \rangle$;
\item Perform the smooth pullback along $\varpi$ of the Chow homology class $c\cup p\cap [W]$ to obtain a class on $\mathbb P_W\langle 1 \rangle$;
\item Apply any polynomial in the Chern class of the hyperplane bundle on $\mathbb P_W\langle 1 \rangle$. 
\item Perform virtual pullback along $\epsilon$.
\item Perform proper pushforward along the  composite morphism $\tau\circ\epsilon$. 
\end{enumerate}

We can absorb the terms $p$ and $c$ into a single cohomological term which we write as $q$. The hyperplane bundle is pulled back from $\mathbb P\langle 1\rangle$, and by applying the projection formula and commutativity of smooth and Gysin pullback, we can obtain the same class instead by \textit{first} calculating the $\varphi\circ\kappa$ pushforward of the class $\varphi^!(q\cap [W])$, pulling back to $\mathbb P\langle 1 \rangle$, applying operators coming from the hyplerplane bundle and pushing forward to $\Mbar_{g,n}\langle 1\rangle$. We will prove that $\varphi^!(q\cap [W])$ pushes forward to a tautological class. 

The virtual stratum $\mathsf{DR}_W(\sV)$ may be written as a product of smaller strata by the gluing formula for strata in Gromov--Witten theory of rubber targets, see for instance~\cite[Section~5.3]{MR14} or~\cite[Section~3]{GV05}. The class $q$ is piecewise polynomial and therefore restricts to a tautological class on each of the relative stable maps spaces. We can now use the result of Faber and Pandharipande, that the pushforwards of tautological Gromov--Witten cycles, including those with disconnected domain curves, to the moduli space of curves are tautological~\cite[Theorem~2]{FP}. We conclude that the virtual strict transform under one blowup is tautological. 
\end{proof}

\subsection{The inductive procedure} The calculations performed for the first blowup guarantee that $\DR(\sV)\langle 1 \rangle$ is equal to the pullback of a tautological class on the moduli space of curves, corrected by classes on strata that come from piecewise polynomials. Practically, these are pushforwards from strata of expressions in the Chern roots of the normal bundles to those strata. 

We now climb up the tower of blowups of $\Mbar_{g,n}$. At each stage, we blowup a smooth center, and we calculate the strict transform exactly as in the previous section. Subtrating off the correction terms at each stage, eventually, we are left with the class $\DR^\lightning(\sV)$. 

\noindent 
{\bf Notation.} The nature of the inductive procedure means that this final section is notation heavy. We will use $\Mbar_{g,n}\langle k\rangle$ to denote the $k^{th}$ blowup of the moduli space of curves, and $t\DR(\sV)\langle k \rangle$ to be the \textit{strict} {transform} of $t\DR(\sV)$ under this blowup. The superscript $(-)^{\mathsf{tot}}$ is used for the \textit{total} transform and $(-)^{\mathsf{exc}}$ is used for objects living over the exceptional divisor of the blowup. 

\noindent
\textsc{The inductive setup.} Let us now describe the inductive step after $k$ blowups have been performed. The setup is as follows:
\[
\begin{tikzcd}
\DR(\sV)\langle k+1\rangle \arrow[hookrightarrow]{r}\arrow{d}\drar[phantom, "\square"] & \DR(\sV)\langle k\rangle^{\mathsf{tot}}\arrow{d}\arrow{r} \drar[phantom, "\square"] & \DR(\sV)\langle k\rangle\arrow{d}{}\\
t\DR(\sV)\langle k+1\rangle \arrow[hookrightarrow]{r}& t\DR(\sV)\langle k\rangle^{\mathsf{tot}}\arrow{d}\arrow{r} \drar[phantom, "\square"] & t\DR(\sV)\langle k\rangle\arrow{d}{\kappa}\\
&\Mbar_{g,n}\langle k+1\rangle\arrow{r}{\nu}& \Mbar_{g,n}\langle k\rangle .
\end{tikzcd}
\]
We know from the earlier steps in the recursion that the class $\DR(\sV)\langle k\rangle$ pushes forward in $\Mbar_{g,n}\langle k\rangle$ to a class that differs from the pullback of a tautological class on the moduli space of curves by classes that come from piecewise polynomials and cotangent classes. This class is pulled back along the next blowup, which a local complete intersection morphism. The pullback is the of the virtual total transform, and we are interested in the virtual strict transform. 

The virtual strict transform, as before, is controlled by Aluffi's Segre class formula. The correction term in the middle row is a sum of classes involving the hyperplane bundle on the exceptional divisor and classes on strata that come from piecewise polynomials.\footnote{We repurpose the notation from the previous section.} By treating the terms one at a time, we may focus on a fixed stratum $W$ of $ t\DR(\sV)\langle k\rangle$. 

\noindent
\textsc{The terms over a fixed stratum.} Restrict the diagram above to a stratum $W$ of $t\DR(\sV)\langle k \rangle$. We have the following diagram, over the exceptional: 
\[
\begin{tikzcd}
\mathsf{DR}_W^{\mathsf{exc}}(\sV)\langle k\rangle\arrow{r}\arrow{d}{\epsilon} \drar[phantom, "\square"] &\mathsf{DR}_W(\sV)\arrow{d}{\varphi}\langle k\rangle \\
\mathbb P_W\langle k\rangle \arrow{d}{\tau}\arrow{r}{g}\drar[phantom, "\square"] & W \arrow{d}{\kappa} \\
\mathbb P\langle k \rangle \arrow{r}{\varpi} & Z\langle k \rangle.
\end{tikzcd}
\]
The horizontal arrows are all projective bundles, and in particular are flat. We recall that the definition of correction term in the blowup: (i) place a piecewise polynomial on $W$, (ii) apply Gysin pullback along $\varphi$, (iii) refined pullback along $\varpi$, and (v) apply the Chern class in the excess normal bundle. 

By compatibility, we can instead apply $\varphi^!$ to an operator on $W$ and push it all the way to $Z\langle k \rangle$, before pulling back to the bundle and only then the excess operator. Therefore we examine the class $\varphi^!(p\cap [W])$ pushed forward to $Z\langle k \rangle$ where $p$ is a piecewise polynomial. There is a sequence of maps
\[
W\hookrightarrow t\DR(\sV)\langle k\rangle \to t\DR(\sV). 
\]
The first map is an inclusion of a stratum. The composite morphism is a partially compactified torus bundle over a stratum in the second codomain. Denote this latter stratum by $W^\flat$ and write the composite morphism above as
\[
W\to W^\flat. 
\]
By hypothesis, the stratum $W^\flat$ is irreducible, so its fundamental class pulls back to the fundamental class of $W$, using the diagonal pullback of~\cite[Section~8.1]{Ful84}. The stratum $W$ is contained in $t\DR(\sV)$ and is therefore equipped with a refined virtual class by virtual pullback along the map $r: \DR(\sV)\to t\DR(\sV)$. The correction is calculated by refined virtual pullback of this virtual class along $W\to W^\flat$, decorating by piecewise polynomials, and pushing forward to $\PP\langle k \rangle$. The diagonal pullback commutes with the virtual pullback along $r: \DR(\sV)\to t\DR(\sV)$, so we focus our attention on the class $r^![W^\flat]$.

\noindent
\textsc{Decomposing the stratum.} The key fact that we will now use is that strata of the double ramification cycle can be expressed as finite quotients of products of smaller double ramification cycle. This allows us to run the induction; let us now provide the details. 

The class $r^![W^\flat]$ is a virtual stratum in a space of relative stable maps, to which we may apply the gluing formula. By the discussion in Section~\ref{sec: tropical-DR} stratum $W^\flat$ corresponds uniquely to a stratum in the moduli space of maps to expansions of the rubber $\mathbb P^1$ geometry, relative to $0$ and $\infty$. A stratum is indexed by a rubber combinatorial type $[\Gamma\to T]$ where as before, $T$ is a line graph and $\Gamma$ is a weighted dual graph. The combinatorial type fixes the ramification orders along all edges and half edges. 

By cutting all the edges of the graph $T$ above, we obtain a collection of tropical maps $[\Gamma_i\to T_i]$ where $T_i$ is a $1$- vertex graph with $2$ legs, with decorations of all edge flags by slopes.  The graph $\Gamma_i$ determines a moduli space $\cM_{\Gamma_i}$ of curves whose genus is the genus label at $\Gamma$ and whose number of markings is equal to the number of flags at this vertex. A graph cover $\Theta = [\Gamma_i\to T_i]$ of this type determines a double ramification cycle problem. 

We obtain a product decomposition, parallel to Section~\ref{sec: proj-bundle}
\[
 W^\flat\to W^\flat_{\Gamma_1}\times\cdots\times W^\flat_{\Gamma_r}/\mathsf{Aut}(\Theta),
\]
exhibiting the stratum as a torus bundle, as we have just explained in Section~\ref{sec: proj-bundle}. These come with numerical specifications $\mathsf{U}_1,\ldots \mathsf{U}_r$ of contact orders for a smaller double ramification problem. We choose any compactification of the map with smooth source and target space to obtain:
\[
\overline W^\flat \to \overline W^\flat_{\Gamma_1}\times \cdots\times\overline W^\flat_{\Gamma_r}/\mathsf{Aut}(\Theta).
\]
The external product on Chow groups gives rise to a \textit{split DR cycle} that is picked out individually on the factors
\[
\DR(\mathsf{U}_1)\times\cdots \times \DR(\mathsf{U}_r)\in \mathsf{CH}_\star(\overline W^\flat;\QQ).
\]
On the other hand, $\overline W^\flat$ also carries a double ramification cycle: the virtual pullback on the stratum $W^\flat$ under the Abel--Jacobi morphism produces a refined class which we pushforward to the compact $\overline W^\flat$. 

%


We come to the final result. Fix a sequence of blowups giving rise to a morphism
\[
\nu_m: \Mbar_{g,n}\langle m \rangle \to \Mbar_{g,n}
\]

\begin{theorem}
The class $[\DR(\sV)\langle m\rangle]$ in the Chow ring of $\Mbar_{g,n}\langle m\rangle$ of the virtual strict transform lies in the subring generated by the pullbacks of tautological classes from the moduli space of curves and Chern roots of the normal bundles to the strata of $\Mbar_{g,n}\langle m\rangle$. In particular, the class $[\DR^\lightning(\sV)]$ is tautological. 
\end{theorem}

\begin{proof}
First note that in genus $0$ the lightning class is uncorrected for all contact data. In all genus, the class is uncorrected when $\sV$ is the $0$ vector or when it is $(1,-1)$. Consider a proper substratum of the double ramification cycle. Each proper stratum of $t\DR(\sV)$ gives rise to smaller double ramification cycle problems -- either the genus or the total number of markings decreases -- and we may inductively assume that all the associated lightning classes are known to be tautological. 

We induct on the number of blowups. We have completed the base case, and come to the inductive step. By the discussion above, the correction class is the pushforward of a sum of classes supported on the exceptional divisor of the blowup. Each of these classes corresponds to a stratum and we call it $W$. As before, we have
\[
\begin{tikzcd}
\mathsf{DR}_W^{\mathsf{exc}}(\sV)\langle k+1\rangle\arrow{r}\arrow{d}{\epsilon} \drar[phantom, "\square"] &\mathsf{DR}_W(\sV)\arrow{d}{\varphi}\langle k+1\rangle \\
\mathbb P_W\langle k+1\rangle \arrow{d}{\tau}\arrow{r}{g}\drar[phantom, "\square"] & W \arrow{d}{\kappa} \\
\mathbb P\langle k+1 \rangle \arrow{r}{\varpi} & Z\langle k+1 \rangle
\end{tikzcd}
\]
and it will suffice for us to show that if we start with a class of the form $p\cap[W]$, apply the virtual pullback along $\varphi$, following by pushforward to $Z\langle k+1 \rangle$, the result is a tautological class. As noted already, the refined class $\varphi^![W]$ is pulled back from the corresponding refined class on $W^\flat\subset t\DR(\sV)$. We compactify, as above
\[
\overline W\to \overline W^\flat
\]
and observe that any sufficiently fine compactification of $\overline W$ will map to $Z\langle k+1\rangle$. It suffices then to show that the following class on $\overline W$ is inductively tautological: (i) compute the refined virtual class on $W^\flat$ obtained via the virtual pullback $r:\DR(\sV)\to t\DR(\sV)$; (ii) pushforward to a sufficiently fine compactification $\overline W^\flat$; (iii) pullback along $\overline W\to \overline W^\flat$; (iv) apply a piecewise polynomial operator $p$. 

The operator $p$ can be arranged to be the restriction of an operator on the compactification, which we continue to denote it $p$. It will therefore suffice to show that the virtual class on $\overline W^\flat$ is tautological. By the discussion in Section~\ref{sec: proj-bundle}, and the weak factorization theorem, we may connect the compactification $\overline W^\flat$ by a sequence of blowups and blow-downs along smooth centers, to the compactification studied in Construction~\ref{compactify-cut}. We recall that this construction outputs a projective bundle,
\[
\mathsf{kut}: \widetilde{\cM}_{[\Gamma\to T]}\to \prod_{u\in V(\Gamma)} \Mbar(u).
\]
The source carries a double ramification class, as it contains $W^\flat$, and we can pushforward the refined virtual class here. As this class is stable under pullbacks along blowups, it suffices to show that the double ramification class on $\widetilde{\cM}_{[\Gamma\to T]}$ is tautological. We apply Proposition~\ref{prop: projective-bundle-tautological}, which reduces the question to the smaller double ramification problems associated to the vertices of $\Gamma$. These are known to be tautological inductively, and we conclude the result. 

\end{proof}

\begin{corollary}
The toric contact cycles lie in the tautological ring of the moduli space of curves. 
\end{corollary}

\begin{proof}
We find a sequence of blowups $\Mbar_{g,n}\langle m \rangle\to \Mbar_{g,n}$ such that all double ramification cycles have been transversalized. By the projective bundle formula and Fulton's key formula~\cite[Thm. 3.3 {\it \&} Prop. 6.7]{Ful84}, under each blowup morphism a polynomial in the Chern roots of the normal bundle to a stratum pushes forward to a class spanned by tautological classes from the moduli space of curves and piecewise polynomial operators applied to the strata. These push forward to tautological classes on the moduli space of curves, and we conclude the main result. 
\end{proof}

\bibliographystyle{siam} 
\bibliography{Intersections}

\end{document}